\documentclass{article}
\usepackage{graphicx} 
\usepackage{biblatex}
\usepackage[margin=1in]{geometry}
\usepackage{amssymb}
\usepackage{amsmath} 
\usepackage{amsthm}
\usepackage{xcolor}
\usepackage{hyperref}
\usepackage{comment}
\usepackage{stackengine}
\usepackage{mathtools}

\usepackage{tikz}
\usepackage{pgfplots}
\pgfplotsset{compat=1.15}
\usepackage{mathrsfs}
\usetikzlibrary{arrows}

\allowdisplaybreaks

\newtheorem{thm}{Theorem}

\newtheorem{lemma}[thm]{Lemma}
\newtheorem{definition}[thm]{Definition}
\newtheorem{prop}[thm]{Proposition}

\newtheorem{remark}[thm]{Remark}
\newtheorem{corollary}[thm]{Corollary}

\mathtoolsset{showonlyrefs}

\title{Spectral convergence of empirical integral operators with discontinuous kernels}
\author{Manuel Dias}
\date{April 2026}

\begin{document}

\maketitle
\begin{abstract}
    We study the spectral behavior as the sample size $n \to +\infty$ of integral operators defined by convolution of a non-negative symmetric kernel k with respect to empirical measures $\mu_n = \frac{1}{n} \sum_{i=1}^n \delta_{X_i}$, where $\{X_i\}_{i=1}^n$ are independent uniform samples from a compact probability metric space $(\mathcal{X},d,\mu)$. Relaxing the usual positivity and continuity assumptions on k, we prove the convergence of these empirical operators to their continuous counterparts, and provide explicit convergence rates.
\end{abstract}
\section{Introduction}

In statistics and computer science, given data points $\{X_1,...,X_n\}$ living on some space $\mathcal{X}$, it is natural to assume that these points are independdently drawn from some probability distribution $\mu$ on $\mathcal{X}$. To understand the structure of the data points it is a common technique to construct a weighted graph by using a kernel function $k: \mathcal{X} \times \mathcal{X} \rightarrow \mathbb{R}$ and to consider the matrices given by
\begin{equation}
\label{eq:basisMatrices}
    K_{n,\mu} = \frac{1}{n}\left(k(X_i, X_j)\right)_{1 \leq i,j\leq n},
    \quad\quad
    (D_{n,\mu})_{i,i} = \frac{1}{n}\sum_{j = 1}^n k(X_i, X_j),
\end{equation}
which we call the similarity matrix and the degree matrix. Many algorithms for clustering, dimension reduction or manifold methods are based on using the eigenvalues and eigenfunctions of matrices that are built from the ones in \eqref{eq:basisMatrices}. These are usually called spectral techniques.

For instance Clustering is a task in statistics and computer science of separating data points $\{X_1,...,X_n\}$ in groups based on some similarity property. Spectral clustering uses the eigenfunctions of Laplace matrices like the unnormalized Graph Laplacian
\begin{equation}
\label{eq:unnormalizedGraphLaplacian}
    D_{n,\mu} - K_{n,\mu}
\end{equation}
or the normalized Graph Laplacians
\begin{equation}
\label{eq:normalizedGraphLaplacian}
    I - D_{n,\mu}^{-\frac{1}{2}}K_n D_{n,\mu}^\frac{1}{2},
\quad\quad
    I - D_{n,\mu}^{-1}K_{n,\mu}.
\end{equation}
to find these groups. Algorithms of that kind can be found in \cite{spectralClusteringAlgorithmAndAnalysis, Kernelspectralclusteringoflargedimensionaldata, vonLuxburg2007, LuxburgSpectralConsistency} for instance. Similarly, spectral data reduction algorithms provides a smaller representation of data $\{X_1,...,X_n\} \subset \mathbb{R}^N$ that lives in a high dimensional Euclidean spaces, by use of the eigenvectors of the Laplace matrix. See for example \cite{BelkinLaplacianEigenmaps, belkinPtWiseConvergence}.

The theoretical justification for these algorithms lies in showing that the spectrum of the finite random operators converges to the spectrum of its continuous integral counterpart as sample size goes to infinity. The first results in this direction were shown in \cite{l2SpectrumConvergence}. They prove that the $\ell^2(\mathbb{N})$ distance between the two ordered spectrum, (since the operators are symmetric compact and bounded the sequence of eigenvalues can be ordered in decreasing order, and this sequence is in $\ell^2(\mathbb{N})$) converges to zero under the assumption that $k$ is symmetric and square integrable. Other similar results about the spectrum of random matrices using related techniques can be found in \cite{belkinLearningWithIntegralOperators}.

In \cite{LuxburgSpectralConsistency} it is proved that eigenvectors of the empirical integral operators converge uniformly to the eigenvectors of the continuous integral operators. This result is obtained by assuming that $k$ is a symmetric continuous kernel that has a positive lower bound and a finite upper bound. The convergence there is not in $L^2(\mathcal{X},\mu)$ but rather in the space of continuous functions $C(\mathcal{X})$. This applies in particular to the kernel $k_t(x,y) = e^{-\|x-y\|^2/4t}$ over compact sets in $\mathbb{R}^N$. The pointwise and spectral properties for this kernel as $t \rightarrow 0$ have been studied in \cite{BelkinEigenmaps, belkinPtWiseConvergence, belkinSingularPtConvergence, susovanDavid}. These are usually related with the Laplace-Beltrami operator of the underlying space when the latter is regular enough to admit such an operator.

Our work aims at generalizing the convergence result of \cite{LuxburgSpectralConsistency}. In fact we will relax the continuity and the positive lower bound of the kernel. In particular we will show a form of the Law of Large Numbers in Theorems \ref{thm:llneigenfunctions1}, \ref{thm:llneigenfunctions2} which show that the eigenfunctions of the empirical operators converge almost surely to the eigenfunctions of the continuous operators. This is done by showing compact convergence, and then applying  \cite[Theorem \ref{thm:funcAnalAbstractSpectralConvergence}]{ChatelinFuncAnalysis1}, which implies some spectral convergence outside the essential spectrum of the operators. We also obtain rates for the convergence of these eigenfunctions in Theorem \ref{thm:quantifiedSpectralConvergence}.

One of the main differences of this work with \cite{LuxburgSpectralConsistency} is the change of domain of the operators. In \cite{LuxburgSpectralConsistency} the authors use the space of continuous functions $\mathcal{C}(\mathcal{X})$, while we use the space of bounded measurable functions $B(\mathcal{X})$ given by \eqref{eq:BanachSpace}.

The interest in the use of discontinuous kernels is natural, since several spectral algorithms may start with some weight $\tilde{k}(x,y)$, and only associate points which are at a given distance $\epsilon > 0$ by defining the new kernel $k(x,y) = \chi_{B_\epsilon(x)}(y) \tilde{k}(x,y)$ (where $\chi_{B_\epsilon(x)}$ is the characteristic function of the ball $B_\epsilon(x)$). This appears for instance in the work by Belkin and Nyogi \cite{BelkinLaplacianEigenmaps} where the Gaussian Graph Laplacian kernel is cut off by the characteristic function of a ball.

    Given a metric measure space $(\mathcal{X},d, \mu)$, under mild conditions on the measure $\mu$, the kernel $k_r(x,y) = \chi_{B_r(x)}(y)$ will satisfy the conditions of our theorem, and our continuous integral operator defined in \eqref{eq:Uoperator} will correspond to the Symmetrized AMV Laplace operator, which is studied in \cite{AKS2, MT2, K, diasAMV}. This convergence result combines well with the results from \cite{diasAMV} which say that if $(\mathcal{X},d) = (M,d_g)$ is a compact Riemannian manifold with boundary, as $r \rightarrow 0$, the spectrum of $\Delta_r$ converges to the spectrum of the Laplace Beltrami operator $C_m \Delta_g$ where $C_m$ is a constant depending on the dimension. Thus if our points $\{X_1,...,X_n\}$ are being uniformly drawn from a manifold then the eigenfunctions from the empirical operator from the graph approximation of the manifold will be close the to the Laplace Beltrami operator eigenfunctions.

\section{Definitions and statement of results}

Let $(\Omega, \mathcal{F}, \mathbb{P})$ be a probability space and $(\mathcal{X}, d, \mu)$ a compact metric measure space. Let $k : \mathcal{X} \times \mathcal{X} \rightarrow [0,\infty[$ be a symmetric positive kernel $k(x,y) = k(y,x)$ and $\left\{X_i : \Omega \rightarrow \mathcal{X}\right\}_{i \in \mathbb{N}}$ a collection of independent random variables distributed uniformly with respect to $\mu$, that is given $A \subset \mathcal{X}$ we have $
    \mathbb{P}(\{X_i \in A\})
=
    \mu(A)
$. In particular $\mu$ is a probability measure since $\mu(\mathcal{X}) = \mathbb{P}(\Omega) = 1$. Given $n \in \mathbb{N}$, let 
\begin{equation}
    \mu_n = \frac{1}{n}\sum_{i=1}^n \delta_{X_i}
\end{equation}
be the empirical measures, and define the functions
\begin{align}
    d_{\mu}(x)
:=&
    \int_{\mathcal{X}}
    k(x,y)
    d\mu(y),\quad\quad\quad\quad\quad\quad\quad\quad
    d_{n,\mu}(x)
:=
    \int_{\mathcal{X}}
    k(x,y)
    d\mu_n(y),\\
\label{eq:definitionOfH}
    h_{\mu}(x,y) 
:=& 
    \frac{1}{2}k(x,y)\left(\frac{1}{d_{\mu}(x)}+\frac{1}{d_{\mu}(y)}\right),
    \quad\quad
    h_{n,\mu}(x,y) 
:= 
    \frac{1}{2}k(x,y)\left(\frac{1}{d_{n,\mu}(x)}+\frac{1}{d_{n,\mu}(y)}\right),\\
    \label{eq:MmuFunction}
    m_\mu(x) 
:=& 
    \int_{\mathcal{X}} h_{\mu}(x,y)d\mu(y),
    \quad\quad\quad\quad\quad\quad\quad
    m_{n,\mu}(x) 
:= 
    \int_{\mathcal{X}} h_{n,\mu}(x,y) d\mu_n(y).
\end{align}
Let $B(\mathcal{X})$ be the Banach space given by
\begin{equation}
\label{eq:BanachSpace}
    B(\mathcal{X}):=
    \{
        f : \mathcal{X} \rightarrow \mathbb{R}
        \text{ measurable such that }
        \sup_{x \in \mathcal{X}}
        \left|
            f(x)
        \right|
        <+\infty
    \},
\end{equation}
with norm given by
\begin{equation}
    \|f\|_{B(\mathcal{X})} = \sup_{x \in \mathcal{X}} |f(x)|.
\end{equation}
This Banach space will be the domain and codomain of the operators considered in this work.

Given a function $g \in B(\mathcal{X})$, consider $M_g :B(\mathcal{X}) \rightarrow B(\mathcal{X})$ to be the multiplication operator 
\begin{equation}
    M_g(f)(x) = g(x)f(x).
\end{equation}

In this work we are interested in the convergence of the empirical operators given by 
\begin{align}
\label{eq:OperatorPnDef}
P_{n, \mu}f(x)
:=&
    \int_{\mathcal{X}}
    k(x,y)
    f(y)
    d\mu_n(y)
=
    \frac{1}{n}
    \sum_{j=1}^n
        k(x,X_i)
        f(X_i),\\
\label{eq:OperatorTnDef}
    T_{n, \mu}f(x)
:=&
\int_{\mathcal{X}}
    h_\mu(x,y)
    f(y)
    d\mu_n(y)
=
    \frac{1}{n}
    \sum_{j=1}^n
    h_\mu(x,X_j)f(X_j),
,\\
\label{eq:THatDef}
    \hat{T}_{n, \mu}f(x)
:=&
\int_{\mathcal{X}}
    h_{n,\mu}(x,y)
    f(y)
    d\mu_n(y)
=
    \frac{1}{n}
    \sum_{j=1}^n
    h_{n,\mu}(x,X_j)
    f(X_j)
,\\
\label{eq:empiricalOperatorUn}
    U_{n, \mu}f(x)
:=&
    \int_\mathcal{X}
        h_{n,\mu}(x,y)
        \left(
            f(x)
        -
            f(y)
        \right)
        d\mu_n(y)=
        M_{m_{n,\mu}}f(x) - \hat{T}_{n, \mu}f(x),\\
\label{eq:empiricalOperatorUn'}
    U_{n, \mu}' :=& I - \hat{T}_{n, \mu},
\end{align}
to the continuous operators given by
\begin{align}
\label{eq:operatorOfKernel}
P_{\mu}f(x)
:=&
    \int_\mathcal{X}
    k(x,y)
    f(y)
    d\mu(y),\\
\label{eq:OperatorTDef}
    T_\mu f(x)
:=&
    \int_{\mathcal{X}}
    \frac{1}{2}k(x,y)
    \left(
        \frac{1}{d_{\mu}(x)}
    +
        \frac{1}{d_{\mu}(y)}
    \right)
    f(y)
    d\mu(y)
=
    \int_{\mathcal{X}}
    h_\mu(x,y)
    f(y)
    d\mu(y),\\
\label{eq:Uoperator}
    U_{\mu}f(x)
:=&
    \int_\mathcal{X}
    \frac{1}{2}
    k(x,y)
    \left(
        \frac{1}{d_{\mu}(x)}
    +
        \frac{1}{d_{\mu}(y)}
    \right)
    \left(
        f(x)
    -
        f(y)
    \right)
    d\mu(y)
=
    M_{m_{\mu}}f(x)
-
    T_\mu f(x),\\
\label{eq:operatorU'}
    U_{\mu}' :=& I - T_\mu.
\end{align}

Throughout the work we assume our kernel $k$ satisfies the following conditions
\begin{equation}
    \label{eq:UpBoundKernel}
    \tag{UB}
    \sup_{x \in \mathcal{X}}\|k(x,\cdot)\|_{B(\mathcal{X})}
    < M < \infty,
\end{equation}
\begin{equation}
\tag{LB}
\label{eq:UpAndLowBound}
0 < a < \inf_{x \in \mathcal{X}} |d_{\mu}(x)|=\inf_{x\in \mathcal{X}}\|k(x,\cdot)\|_{L^1(\mathcal{X})},
\end{equation}
for some $0<a\leq 1$ and $1 \leq M < \infty$. To generalize continuity, given $\delta > 0$ define
\begin{equation}
    k_\delta(x; z)
:=
    \sup_{y \in B_\delta(x)}
        |
        k(x,z)
    -
        k(y,z)
        |
\end{equation}
and assume the kernel also satisfies
\begin{equation}
\tag{C}
\label{eq:maxDifferenceInBall}
    \|
        k_\delta(x; \cdot)
    \|_{L^1(\mathcal{X})}
\leq
    \omega(\delta),
\end{equation}
where $\omega: [0,+\infty] \rightarrow \mathbb{R}$ is a continuous increasing function that satisfies $\lim_{\delta \rightarrow 0}\omega(\delta) = 0$. Note that $k_\delta(x;z)$ is not necessarily symmetric. Conditions \eqref{eq:UpBoundKernel} and \eqref{eq:maxDifferenceInBall} do not depend only on $k$ but rather they are a relation between the kernel $k$ and the measure $\mu$. In fact given a kernel $k:\mathcal{X} \times \mathcal{X} \rightarrow \mathbb{R}$, a continuous incresing function $w(\cdot)$ such that $\lim_{\delta \rightarrow 0}\omega(\delta) = 0$ and a value $a>0$ we can consider the set of measures 
\begin{align} 
\label{eq:classOfMeasures}
    \mathcal{U}_{a, \omega(\cdot), k}
:=
    \Big\{
        \mu \in \mathcal{M}(\mathcal{X}) :
        &\text{ } \mu \text{ is probability measure on } \mathcal{X} \text{ and satisfies equations }
        \eqref{eq:UpAndLowBound}, \eqref{eq:maxDifferenceInBall}\\
        &
        \text{ with constant } a \text{ and function } \omega(\cdot) \text{ respectively with the kernel } k
    \Big\}.
\end{align}

We now present the main results. First we present the Laws of Large Numbers for the convergence of the spectrum, and then the result with rates of convergence. For these results we refer to the definitions of convergence \ref{def:convergenceDefinitions} and of spectral projections \ref{def:spectralProjections}.

\begin{thm}
\label{thm:llneigenfunctions1}
    Let $\lambda \neq 1$ be an eigenvalue of $U_\mu'$ and $V \subset \mathbb{C}$ an open set such that $V \cap \sigma(U_\mu') = \{\lambda\}$. Then:
    \begin{enumerate}
        \item $U_{n, \mu}' \stackunder{$\longrightarrow$}{$\scriptsize{\text{c}}$} U_\mu'$ compactly almost surely.
        \item The eigenvalues in $\sigma(U_{n, \mu}') \cap V$ converge to $\lambda$ almost surely.
        \item If $\textup{Pr}_n$ is the spectral projection of $\sigma(U_{n, \mu}') \cap  V$ and $\textup{Pr}$ is the spectral projection of $\sigma(U_{\mu}) \cap V$, then $\textup{Pr}_n \stackunder{$\longrightarrow$}{$\scriptsize{\text{p}}$} \textup{Pr}$ pointwise almost surely. This convergence is in the uniform topology of functions.
    \end{enumerate}
\end{thm}

\begin{thm}
\label{thm:llneigenfunctions2}
    Let $\lambda \notin rg(m_{\mu})$ be an eigenvalue of $U_\mu$ and $V \subset \mathbb{C}$ an open set such that $V \cap \sigma(U_\mu) = \{\lambda\}$. Then:
    \begin{enumerate}
        \item $U_{n, \mu} \stackunder{$\longrightarrow$}{$\scriptsize{\text{c}}$} U_{\mu}$ converges compactly almost surely.
        \item The eigenvalues in $\sigma(U_{n, \mu}) \cap V$ converge to $\lambda$ almost surely.
        \item If $\textup{Pr}_n$ is the spectral projection of $\sigma(U_{n, \mu}) \cap  V$ and $\textup{Pr}$ is the spectral projection of $\sigma(U_\mu) \cap V$, then $\textup{Pr}_n \stackunder{$\longrightarrow$}{$\scriptsize{\text{p}}$} \textup{Pr}$ pointwise almost surely. This convergence is in the uniform topology of functions.
    \end{enumerate}
\end{thm}

\begin{thm}
\label{thm:rateOfConvergenceTheorem}
    Let $u \in B(\mathcal{X})$ be an eigenvector of $U_\mu'$ with eigenvalue $\lambda \neq 1$, and $V \subset \mathbb{C}$ an open subset such that $V \cap \sigma(U_\mu) = \{\lambda\}$. Let $\text{Pr}_n$ be the spectral projection of $\sigma(U'_{n,\mu}) \cap V$.
    Let $L_\delta = N(\mathcal{X}, d, \delta)$ be the minimum number of balls of radius $\delta$ needed to cover $\mathcal{X}$.
    Suppose that there exist $C_L,C_\omega, m, m' > 0$ such that 
    \begin{align}
    \label{eq:hypOnCoverAndContinuity}
        L_\delta 
        \leq 
        C_L\frac{1}{\delta^{m}},\quad
        w(\delta) 
        \leq 
        C_\omega\delta^{m'}.
    \end{align}
    and 
    let {$\tilde{\gamma} = M^2\gamma/a$} with $\gamma = 32M^2/a^2 + \frac{8}{3}M/a$. Let $\alpha \geq 1$ such that
    \begin{equation}
        \frac{\alpha}{\tilde{\gamma}}
        -
        \frac{m}{2m'}
        >
        1.
    \end{equation}
    Then there exists $\overline{C}_{\lambda, T_\mu}>0$ such that for all $N \in \mathbb{N}$ with $N \geq 2$
    \begin{align}
    &\mathbb{P}
    \left(
    \bigcap_{n \geq N}
    \left\{
        \|u - \text{Pr}_n u\|_{B(\mathcal{X})}
    \leq
        4\overline{C}_{\lambda, T_\mu}\frac{\alpha}{\sqrt{n}}\sqrt{\ln(n)}
    \right\}
    \right)\\
    \geq&
    1
    -\frac{16\tilde{C}}{\alpha^{\frac{2m}{m'}}}
    \frac{\tilde{\gamma}m'}
    {\alpha m' - \tilde{\gamma} m - \tilde{\gamma}m'}
    \frac{1}{(N-1)^{\frac{\alpha}{\tilde{\gamma}}-\frac{m}{2m'}}}
    +
    4\frac{C_a}{C_e}
    \exp
    \left\{
        -(N-1)C_e
    \right\}
    \end{align}
    where $\tilde{C} := C_L\left(\frac{8MC_\omega}{a \ln(2) }\left(\frac{1}{a}+\frac{2M}{a} \right) \right)^\frac{m}{m'}$, $C_a := C_L
            \left(
            \frac{16C_\omega}{a}
            \right)^\frac{m}{m'}$ and $C_e := \frac{(a/2)^2}{32M^2 + \frac{8}{3}M(a/2)}$.
\end{thm}

\begin{remark}
\label{rem:remarkAboutAnulli}
    These assumptions are strongly influenced by the case where the kernel $k$ is given by
\begin{equation}
    k_r(x,y) = \chi_{B_r(x)}(y).
\end{equation}
In this situation we have that
\begin{align}
    \|k_r(x,\cdot)\|_{L^1(\mathcal{X})} 
= 
    \mu(B_r(x)),
    \quad\quad
    \|(k_r)_\delta(x;\cdot)\|_{L^1(\mathcal{X})}
=
    \mu(B_{r+\delta}(x) \backslash B_{r-\delta}(x)).
\end{align}
And in very generic situations (for instance length spaces with a doubling Borel regular measure, see \cite[Corollary 2.2]{Buckley1999}, \cite[Lemma 2.1]{TomaszMeasureZeroSphere}), it is natural that 
\begin{equation}
    \lim_{\delta \rightarrow 0}
    \sup_{x \in \mathcal{X}}
    \mu(B_{r+\delta}(x)\backslash B_{r-\delta}(x))
=
    0,
    \quad\quad
    \mu(B_r(x))=
    \|k_r(x,\cdot)\|_{L^1(\mathcal{X})}>a_r>0,
\end{equation}
for some $a_r>0$.
\end{remark}

\begin{remark}
    As already mentioned, one of the main differences of this work with \cite{LuxburgSpectralConsistency} is the change of domain of the operators. This is so because even if $f \in \mathcal{C}(\mathcal{X})$, it is not necessarily true that $T_{n, \mu} f$, $\hat{T}_{n, \mu}f$, $U_{n, \mu} f$ or $U_{n,\mu}'f$ are in $\mathcal{C}(\mathcal{X})$, because $k(x,X_i)$ is not necessarily a continuous function of $x$. Interestingly, even if $f \in B(\mathcal{X})$, condition \eqref{eq:maxDifferenceInBall} implies that $P_{\mu}f, T_{\mu} f \in \mathcal{C}(\mathcal{X})$, and in particular that the eigenvectors of $T_{\mu}, U_{\mu}$ and $U_{\mu}'$ are continuous. Since we want to use spectral results for Banach spaces which require the domain to be the same as the codomain, we use the space $B(\mathcal{X})$. We also use the space $B(\mathcal{X})$ instead of $L^\infty(\mathcal{X},\mu)$ because we want our functions to be defined everywhere, and in $L^\infty(\mathcal{X},\mu)$ one has to deal with equivalence classes of functions that are defined almost everywhere. We don't use $l^\infty(\mathcal{X})$ (where $l^\infty(\mathcal{X})$ is the space of all bounded functions, that is $L^\infty$ with the counting measure), as this might contain functions that are not measurable and so the integral is not well defined.   
\end{remark}

\subsubsection*{Acknowledgments} The
author is supported by the Research Foundation– Flanders (FWO) via the Odysseus
programme Geometric and analytic properties of metric measure spaces with spectral
curvature constraints, with applications to manifold learning (G0DBZ23N).

\section{Functional analysis preliminaries}

The objective of this work will be to show some spectral convergence of operators. We will not prove directly the spectral convergence of the operators, but will prove some other convergence which in turn implies the spectral convergence. For this we need some results from spectral and perturbation theory.

Given a Banach space $(V, \|\cdot\|_V)$ and a bounded operator $S: V \rightarrow V$ let $\sigma(S) \subset \mathbb{C}$ be its spectrum. We denote $\sigma_{\text{d}}(S)$ to be the discrete spectrum, consisting of isolated eigenvalues with finite dimensional eigenspaces. We define the essential spectrum of an operator $S$ to be $\sigma_{\text{ess}}(S) = \sigma(S) - \sigma_{\text{d}}(S)$. One interesting fact about the essential spectrum is that it is invariant under compact perturbation (see for instance \cite{Kato1995}). That is if $K : V \rightarrow V$ is a compact operator, then $\sigma_{\text{ess}}(S) = \sigma_{\text{ess}}(S + K)$. The essential spectrum will be important to us since our convergence results for the spectrum only hold for the discrete spectrum. In particular the operator $U_\mu$ from equation \eqref{eq:Uoperator} will have a substantial part of its spectrum in the essential spectrum, only allowing the convergence results for the eigenvalues and eigenvectors outside this set.

We start by defining some notions of convergence which will be of interest to us.

\begin{definition}
\label{def:convergenceDefinitions}
    Let $(V, \|\cdot\|_V)$ be a Banach space and let $S_n: V \rightarrow V$ and $S:V \rightarrow V$ a sequence of bounded operators.
    \begin{enumerate}
        \item $S_n$ converges pointwise to $S$, denoted $S_n \stackunder{$\longrightarrow$}{$\scriptsize{\text{p}}$} S$ if for all $x \in V$ we have $\|S_nx- Sx\|_V \rightarrow 0$.
        \item $S_n$ converges in norm to $S$, denoted $S_n \stackunder{$\longrightarrow$}{\scriptsize{${\|\cdot\|}$}} S$ if $
        \|S-S_n\| := \sup_{x \in V \backslash \{0\}} {\|(S-S_n)x\|_V}/{\|x\|_V} \rightarrow 0$.
        \item $S_n$ converges compactly to $S$, denoted $S_n \stackunder{$\longrightarrow$}{\scriptsize$c$} S$, if for every sequence $x_n \in V$ such that $\|x_n\|_V \leq 1$, then $(S-S_n)x_n$ has a convergent subsequence.
        \item $S_n$ is collectively compact if $\cup_{\{n\geq N\}} S_n\left(\{x\in V: \|x\| \leq 1\}\right)$ has compact closure for some $N \in \mathbb{N}$.
        \item $S_n$ converges collectively compactly to $S$, denoted by $S_n \stackunder{$\longrightarrow$}{$\scriptsize \text{cc}$} S$, if $S_n \stackunder{$\longrightarrow$}{$\scriptsize{\text{p}}$} S$, and $S_n$ is collectively compact.
    \end{enumerate}
\end{definition}

\begin{definition}
\label{def:spectralProjections}
    Given a Banach space $(V, \|\cdot\|_V)$ and an operator $S:V \rightarrow V$. Given a simple curve $\gamma : [0,1] \rightarrow \mathbb{C}$ let $\Gamma_{\text{int}}$ be the interior of the curve. Given $W \subset \mathbb{C}$, suppose that $\Gamma_{\text{int}} \cap \sigma(S) = W \cap \sigma(S)$ and that $\gamma([0,1])\cap \sigma(S) = \emptyset$. Then we call the operator $\text{Pr}_{S,W}: V \rightarrow V$ given by
    \begin{equation}
        \text{Pr}_{S, W}
    :=
        \int_{\gamma([0,1])}
        \frac{1}{S-z I}
        dz,
    \end{equation}
    the spectral projection onto the eigenspace $\sigma(S) \cap W$.
\end{definition}
For a reference on this result see \cite{Kato1995}. Notice that in a lot of situations we will have that $W \cap \sigma(S)$ will correspond to a set of eigenvalues of finite multiplicity, and in this case, the operator will  correspond to the sum of projections into the eigenspaces.

The first theorem that we state shows that compact convergence of operators implies some form of spectral convergence. This result can be found in \cite{convergenceFAChatelin}.

\begin{thm}
\label{thm:funcAnalAbstractSpectralConvergence}
    Let $(V, \|\cdot\|_V)$ be a Banach space and $S_n$ and $S$ bounded linear operators on $V$ such that $S_n \stackunder{$\longrightarrow$}{\scriptsize $c$} S$. Consider $\lambda \in \sigma_{\textup{d}}(S)$ an isolated eigenvalue with finite multiplicity $m$, and $W \subset \mathbb{C}$ a neighborhood such that $\sigma(S) \cap W = \{\lambda\}$. Then:
    \begin{enumerate}
        \item There exists $N$ such that for all $n > N$ we have that $\sigma(S_n) \cap W$ is an isolated eigenvalue of $S_n$ with multiplicity $m$. Also the sequence of eigenvalues $\sigma(S_n) \cap W = \{\lambda_n\}$ converges to $\lambda$.
        \item Let $\textup{Pr}$ be the spectral projection of $S$ onto the eigenspace corresponding to $\lambda$. Let $\textup{Pr}_n$ be the projections onto the eigenspace of $S_n$ corresponding to the eigenvalue in $\sigma(S_n)\cap W = \{\lambda_n\}$. Then we have
        \begin{equation}
            \textup{Pr}_{n}
            \stackunder{$\longrightarrow$}{\scriptsize $p$}
            \textup{Pr}.
        \end{equation}
        In particular if the eigenvalue is simple, we have the convergence of the eigenvectors up to a change of sign.
    \end{enumerate}
\end{thm}

The compact convergence does indeed imply some form of spectral convergence, however to obtain rates of convergence for these we will need a result found in the work by Atkinson \cite{Atkinson}, which needs collectively compact convergence.
\begin{thm}
\label{thm:quantifiedSpectralConvergence}
    Let $(V, \|\cdot\|_V)$ be a Banach space and $S_n$ and $S$ compact linear operators on $V$ such that $S_n \stackunder{$\longrightarrow$}{\scriptsize $cc$} S$. Consider $\lambda \in \sigma_{\textup{d}}(S)$ an isolated eigenvalue with finite multiplicity $m$, and $W \subset \mathbb{C}$ a neighborhood such that $\sigma(S) \cap W = \{\lambda\}$. Also denote the spectral projection onto the eigenvalue $\lambda$ by $\textup{Pr}$ and suppose $x$ is an eigenvector of $S$ with eigenvalue $\lambda$. Then there exists $N$ such that for $n > N$ the set $\sigma(S_n) \cap W$ is an isolated point. Let $\textup{Pr}_n$ be the projection corresponding to $\sigma(S_n) \cap W$. Then there exists a constant $C_0>0$ such that 
    \begin{equation}
    \label{eq:estimationOfEigenfunctions}
        \|x - \textup{Pr}_n x\|_V
    \leq
        C_0\left(
            \|
                \left(S_n - S\right)x
            \|_V
        +
            \|x\|_V
            \|
                (S-S_n)S_n
            \|
        \right),
    \end{equation}
    where $C_0$ is independent of $x$ but depends on $\lambda$ and $\sigma(S)$.
\end{thm}
Finally there are two other concepts that are usually defined in the space of continuous functions, but we are going to extend them to the space of functions $B(\mathcal{X})$. 
\begin{definition}
Given the i.i.d variables $\{X_n\}_{n \in \mathbb{N}}$ uniformly distributed with respect to $\mu$ and $\mu_n = \frac{1}{n}\sum_{i=1}^n \delta_{X_i}$, we say that a collection $\mathcal{F} \subset B(\mathcal{X})$ is a Glivenko-Cantelli class if
\begin{equation}
\mathbb{P}
\left(
\left\{
    \limsup_n
    \sup_{f \in \mathcal{F}}
    \left|
        \int_{\mathcal{X}} f(y) d\mu(y)
    -
        \int_{\mathcal{X}} f(y) d\mu_n(y)
    \right|
    = 0
\right\}
\right)
=
1
\end{equation}
\end{definition}
\begin{definition}
\label{def:coveringNumbers}
Given $\epsilon > 0$ and a metric space $(\mathcal{F},d_{\mathcal{F}})$, we define the covering number $N(\mathcal{F},\epsilon, d_{\mathcal{F}})$ as the smallest number $n$ such that $\mathcal{F}$ can be covered by $n$ balls of radius $\epsilon$.
\end{definition}
Notice that the number above is bounded for all $\delta>0$ when $(\mathcal{F},d_{\mathcal{F}})$ is totally bounded. These notions can be found in \cite{weakConvergenceStatistics, Mendelson2003}. Given $g \in B(\mathcal{X})$, we will be particularly interested in studying the following classes of functions
\begin{equation}
    \mathcal{K} = \{k(x,\cdot) : x \in \mathcal{X}\},
    \quad\quad
    \mathcal{H} = \{h_\mu(x,\cdot) : x   \in \mathcal{X}\},
\end{equation}
\begin{equation}
\label{eq:familygK}
    g\cdot \mathcal{K}
=
    \{
        g(\cdot)k(x,\cdot) : x   \in \mathcal{X}
    \},\quad\quad
    g\cdot \mathcal{H}
=
    \{
        g(\cdot)h_\mu(x,\cdot) : x   \in \mathcal{X}
    \},
\end{equation}
\begin{equation}
    \mathcal{H}
    \cdot
    \mathcal{H}
=
    \{
        h_\mu(x,\cdot)
        h_\mu(z,\cdot)
        :
        x,z \in \mathcal{X}
    \}.
\end{equation}

In the next proposition we study some functional analysis properties of the operators $P_\mu$ and $T_\mu$. In particular we show that their image is made of continuous functions.

    \begin{prop}
        The space $(B(\mathcal{X}), \|\cdot\|_B(\mathcal{X}))$ is a Banach space.
    \end{prop}
    \begin{proof}
        We only need to show that Cauchy sequences converge, since $\|\cdot\|_{B(\mathcal{X})}$ is a norm. Thus suppose that $f_n$ is a sequence that is Cauchy for $\|\cdot\|_{B(\mathcal{X})}$. This implies that for each $x \in \mathcal{X}$ the sequence $f_n(x) \in \mathbb{R}$ is Cauchy, and so the sequence $f_n(x)$ converges in $\mathbb{R}$. Let $f(x)$ be it's limit. This implies that $f$ is the pointwise limit of $f_n$, and since $f_n$ are measurable this implies that $f$ is measurable. Since $f_n$ is Cauchy, there exist $N,K \in \mathbb{N}$ such that for all $n \geq N$ and $k \geq K$ we have
        \begin{equation}
            \|f_n(x)-f_m(x)\|_{B(\mathcal{X})}
        <
            1.
        \end{equation}
        Thus given any $x \in \mathcal{X}$ there exists $n_{x} \geq N$ such that
        \begin{equation}
            |f_{n_{x}}(x)-f(x)|
        \leq
            1.
        \end{equation}
        This implies that
        \begin{equation}
            |f(x)|
        \leq
            |f_{n_{x}}(x)-f(x)|
        +
            |f_{n_{x}}(x)-f_K(x)|
        +
            |f_K(x)|
        \leq
            2+\|f_K\|_{B(\mathcal{X})}.
        \end{equation}
        Thus $f \in B(\mathcal{X})$. Now we show that $f_n$ converges in $B(\mathcal{X})$ to $f$. Given $\epsilon > 0$, there exists $N,K \in \mathbb{N}$ such that for $n \geq N, k\geq K$ we have for all $x \in \mathcal{X}$
        \begin{equation}
            |f_k(x)-f_n(x)|
        <
            \epsilon.
        \end{equation}
        Taking the limit $k \rightarrow \infty$, we obtain for all $x \in \mathcal{X}$ and $n \geq N$
        \begin{equation}
            |f(x)-f_n(x)|
        \leq
            \epsilon,
        \end{equation}
        that is
        \begin{equation}
            \|f-f_n\|_{B(\mathcal{X})}
        \leq
            \epsilon
        \end{equation}
        for all $n \geq N$, finishing the proof.
    \end{proof}
    
    \begin{prop}
    \label{prop:imageInContinuousFunction}
        Let $a > 0$ and $\omega(\cdot): [0, +\infty[ \rightarrow \mathbb{R}$ a function such that $\lim_{\delta \rightarrow 0} \omega(\delta) = 0$. Given $\mu \in \mathcal{U}_{a,\omega(\cdot), k}$ then the operators $P_\mu$ and $T_\mu$ are compact, and have their image contained in $C(\mathcal{X})$, the space of continuous functions.
    \end{prop}
    \begin{proof}
        We prove these properties for $P_\mu$ arguing they are the same for $T_\mu$ using Lemma \ref{lem:estimatesForH}, that is the fact that $T_\mu$ can be viewed as an operator $P_\mu$ with kernel $h_\mu$ having the same properties as $k$.  Given $g \in B(\mathcal{X})$, to prove the fact that $P_\mu g$ is continuous, given $x,x'$ such that $x' \in B_\delta(x)$ then using equation \eqref{eq:maxDifferenceInBall}
        \begin{align}
            \left|
                P_\mu g(x)
            -
                P_\mu g(x')
            \right|
        &=
            \left|
            \int_{\mathcal{X}}
            \left(
                k(x,y)g(y)
            -
                k(x',y)g(y)
            \right)
            d\mu(y)
            \right|\\
        &\leq
            \|
                g
            \|_{B(\mathcal{X})}
            \int_{\mathcal{X}}
                \sup_{z \in B_\delta(x)}
                \left|
                    k(x,y)
                -
                    k(z,y)
                \right|
                d\mu(y)\\
        &\leq
        \label{eq:estimationOfDiffForP}
            \|
                g
            \|_{B(\mathcal{X})}\omega(\delta),
        \end{align}     
        and $\omega(\delta) \rightarrow 0$ as $\delta \rightarrow 0$. This shows that $P_\mu(B(\mathcal{X})) \subset C(\mathcal{X})$. On the other hand given a sequence $g_n \in B(\mathcal{X})$ such that $\|g_n\|_{B(\mathcal{X})} \leq 1$, we show that $P_\mu g_n \in C(\mathcal{X})$ is uniformly bounded and is equicontinuous. Thus since $\mathcal{X}$ is compact, by  Arzelà–Ascoli Theorem, there exists a subsequence such that $P_\mu g_n$ converges in $C(\mathcal{X})$. For the uniform bound we use equation \eqref{eq:UpBoundKernel}
        \begin{equation}
            \left|
                P_\mu g_n(x)
            \right|
        \leq
            \left|
            \int_\mathcal{X}
                k(x,y)g_n(y)
            d\mu(y)
            \right|
        \leq
            \|g_n\|_{B(\mathcal{X})}
            M
        \leq
            M.
        \end{equation}
        For the equicontinuity, let $\epsilon > 0$ and $\delta > 0$ small enough such that $\omega(\delta) < \epsilon$. Then given $x \in \mathcal{X}$ and $x' \in B_\delta(x)$ using the same estimate as in inequality \eqref{eq:estimationOfDiffForP} we have
        \begin{align}
            \left|
                P_\mu g_n(x)
            -
                P_\mu g_n(x')
            \right|
        &\leq
            \|
                g_n
            \|_{B(\mathcal{X})}
            \omega(\delta)
    \leq
        \epsilon,
        \end{align}
        showing the equicontinuity, finishing the proof.
    \end{proof}

\begin{corollary}
     Let $u \in B(\mathcal{X})$. If $u$ is an eigenvector of the operators $P_\mu$ or $T_\mu$ with eigenvalue $\lambda \neq 0$, then $u \in C(\mathcal{X})$. If $u$ is an eigenvector of $U_\mu'$ with eigenvalue $\lambda \neq 1$ then $u \in C(\mathcal{X})$. If $u$ is an eigenvector of $U_\mu$, \eqref{eq:Uoperator} with eigenvalue $\lambda \notin rg(m_\mu)$ then $u \in C(\mathcal{X})$.
\end{corollary}
\begin{proof}
    Using Proposition \eqref{prop:imageInContinuousFunction}, by the eigenvalue equation, if $\lambda \neq 0$ and $u \in B(\mathcal{X})$ is an eigenfunction of $P_\mu$ 
    then $u = \frac{1}{\lambda} P_\mu u \in C(\mathcal{X})$. The same works for $T_\mu$. For $U_\mu'$ if $u \in B(\mathcal{X})$ is an eigenvector of $U_\mu'$ with eigenvalue $\lambda \neq 1$, then
    \begin{equation}
        \lambda u = U_\mu' u = u - T_\mu u,
    \end{equation}
    that is $u = \frac{1}{\lambda - 1} T_\mu u \in C(\mathcal{X})$. Finally by Proposition \eqref{prop:imageInContinuousFunction}, we have that $d_\mu(x) = P_\mu 1(x)$ is a continuous function, and so the function $m_\mu$ given by \eqref{eq:MmuFunction} is also continuous by \eqref{eq:UpAndLowBound}. Thus if $\lambda \notin rg(d_\mu)$, and $u \in B(\mathcal{X})$ is an eigenvector of $U_\mu$ with eigenvalue $\lambda$, we have
    \begin{equation}
        u(x)
    =
        \frac{1}{\lambda - m_\mu(x)}
        T_\mu u(x)
        \in C(\mathcal{X}),
    \end{equation}
    concluding the proof.
\end{proof}
Notice that even though the eigenvectors of $T_{\mu,n}$, $U_{\mu,n}$ and $U_{\mu,n}'$ might not be continuous, by the spectral convergence they will convergence uniformly to continuous functions, which guarantees that the discontinuities of the eigenvectors tend to disappear as $n \rightarrow \infty$.

\section{Relation between discrete and continuous operators}

Consider a sequence of i.i.d variables $\{X_1,...,X_n\} \subset \mathcal{X}$ uniformly distributed with respect to $\mu$. In statistical analysis one deals with the empirical $n \times n$ kernel matrix given by
\begin{equation}
    K_{n,\mu} = \left(k(X_i, X_j)\right)_{1 \leq i,j\leq n}
\end{equation}
In some algorithms (like spectral clustering with random-walk Laplacian or data reduction algorithms) one must use the eigenvectors and eigenvalues of the matrices $K_{n,\mu}$, or those from the unnormalized Graph Laplacian \eqref{eq:unnormalizedGraphLaplacian} or the normalized Graph Laplacian \eqref{eq:normalizedGraphLaplacian}. However the matrices we study will not be any of these. We introduce a third matrix $M_{n,\mu}$, which is a diagonal, given by
\begin{equation}
    (M_{n,\mu})_{i,i}
=   
    \frac{1}{n}
    \sum_{j=1}^n
    \left( 
        \frac{k(X_i, X_j)}
        {d_{n,\mu}(X_j)}
    \right).
\end{equation}
With this we can define the matrices
\begin{equation}
\label{eq:empiricalMatrixAMV}
    L_{n,\mu} = 
    \frac{1}{2}
    \left(
    I + M_{n,\mu} - D_{n,\mu}^{-1}K_{n,\mu} - K_{n,\mu}D_{n,\mu}^{-1}
    \right),
\end{equation}
\begin{equation}
\label{eq:empiricalMatrixIdentity}
    L'_{n,\mu} = I - \frac{1}{2}\left( D_{n,\mu}^{-1}K_{n,\mu} + K_{n,\mu}D_{n,\mu}^{-1}\right).
\end{equation}
These matrices will be the ones associated with the empirical operators $U_{n,\mu}$ and $U_{n,\mu}'$ respectively, the same way $K_{n,\mu}$ is associated with $P_{n,\mu}$. One of the initial difficulties in relating these operators seems to be that they are defined in different spaces. By relating these matrices with the operators, we can understand the question of convergence of the matrix operators in the space of operators of $B(\mathcal{X})$. We do this by defining a projection operator $\rho_n : B(\mathcal{X}) \rightarrow \mathbb{R}^n$ given by
\begin{equation}
    \rho_n f = \left( f(X_1),..., f(X_n)\right).
\end{equation}
The following Lemma describes the relations between these operators
\begin{lemma}
    \label{lem:MatrixToOperator}
    We have that
    \begin{align}
        \rho_n \circ P_{n, \mu} &= \frac{1}{n}K_{n,\mu} \circ \rho_n,\\
        \rho_n \circ U_{n, \mu}  &= L_{n,\mu} \circ \rho_n,\\
        \rho_n \circ U_{n, \mu}'&= L'_{n,\mu} \circ \rho_n.
    \end{align}
\end{lemma}

This result establishes the relation between the spectrum of the finite dimensional operators $K_{n,\mu}$, $L_{n,\mu}$, $L'_{n,\mu}$ and their infinite dimensional counter-parts $P_{n, \mu}$, $U_{n, \mu}$ and $U_{n, \mu}'$. Now we make a characterization of the spectrum of both $U_{n, \mu}$ and $U_{n, \mu}'$ in the following lemmas.

\begin{lemma}
\label{lem:discreteFiniteTransition}
    Suppose that $\inf_{x} d_{n,\mu}(x) > 0$. Then we have the following properties for $U_{n, \mu}$ and $L_{n,\mu}$:
    \begin{enumerate}
        \item If $f \in B(\mathcal{X})$ is an eigenvector of $U_{n, \mu}$ with eigenvalue $\lambda$, then $\rho_n f \in \mathbb{R}^n$ is an eigenvector of $L_{n,\mu}$ with eigenvalue $\lambda$.
        \item If $f \in B(\mathcal{X})$ is an eigenvector of $U_{n,\mu}$ with eigenvalue $\lambda \notin rg(m_{n,\mu})$ and $v:=(v_1,...,v_m) = \rho_n f \in \mathbb{R}^n$, then $f$ has the form
        \begin{equation}
        \label{eq:constructionOfSolution2}
            f(x) =
                \frac{1}{n}
                \frac{\sum_{j=1}^n
                h(X_j, x)v_j}{m_{n,\mu}(x) - \lambda}.
        \end{equation} 
        \item If $v \in \mathbb{R}^n$ is an eigenvector of $L_{n,\mu}$ with eigenvalue $\lambda \notin rg(m_{n,\mu})$, and $f \in B(\mathcal{X})$ has the form of equation \eqref{eq:constructionOfSolution2}, then it is an eigenfunction of $U_{n, \mu}$ with eigenvalue $\lambda$.
        \item The essential spectrum of $U_{n, \mu}$ is $\sigma_{\text{ess}}(U_{n, \mu}) \subset rg(m_{n,\mu})$. Similarly $\sigma_{\text{ess}}(U_{\mu}) \subset rg(m_\mu)$.
    \end{enumerate}
\end{lemma}
\begin{proof}
    \textit{1.} Follows from Lemma \ref{lem:MatrixToOperator}.

    \textit{2.} Since the eigenvalue equation $U_{n,\mu}f = \lambda f$ yields $(M_{m_{n,\mu}} - \lambda)f = \hat{T}_{n,\mu} f$.

    \textit{3.} Is obtained by evaluating $U_{n,\mu}$ and using the fact that $v$ is an eigenfunction of $L_{n,\mu}$.

    \textit{4.} Use the fact that $U_{n,\mu} = M_{m_{n,\mu}} - \hat{T}_{n,\mu}$ and $\hat{T}_{n,\mu}$ is a compact operator (check Lemma \ref{lem:finiteOperatorIsCompact}). Thus by  the invariance of the essential spectrum by sums of compact operators we have
    \begin{equation}
        \sigma_{\text{ess}}(M_{m_{n,\mu}} - \hat{T}_{n,\mu}) = \sigma(M_{m_{n,\mu}}) \subset rg(m_{n,\mu})
    \end{equation}
\end{proof}

The condition that $\inf_{x} d_{n,\mu}(x) > 0$ is imposed to guarantee that the operators $\hat{T}_{n,\mu}$ are well defined.

\begin{lemma}
    Suppose that $\inf_{x} d_{n,\mu}(x) > 0$. Then we have the following properties for $U_{n, \mu}'$ and $L_{n,\mu}'$:
    \begin{enumerate}
        \item If $f \in B(\mathcal{X})$ is an eigenvector of $U_{n, \mu}'$ with eigenvalue $\lambda$, then $\rho_n f \in \mathbb{R}^n$ is an eigenvector of $L_n'$ with eigenvalue $\lambda$.
        \item If $f \in B(\mathcal{X})$ is an eigenvector of $U_{n,\mu}'$ with eigenvalue $\lambda \neq 1$ and $v:=(v_1,...,v_n)=\rho_nf \in \mathbb{R}^n$, then $f$ has the form
        \begin{equation}
        \label{eq:constructionOfSolution}
            f(x)
            =
                \frac{1}{n}
                \frac{h_{n,\mu}(X_j,x)v_j}{1-\lambda}.
        \end{equation} 
        \item If $v \in \mathbb{R}^n$ is an eigenvector of $L_{n,\mu}'$ with eigenvector $\lambda \neq 1$, and $f \in B(\mathcal{X})$ has the form of equation \eqref{eq:constructionOfSolution}, then it will be an eigenfunction of $U_{n, \mu}'$ with eigenvalue $\lambda$.
        \item The essential spectrum of $U_{n, \mu}'$ is $\sigma_{\text{ess}}(U_{n, \mu}) = \{1\}$
    \end{enumerate}
\end{lemma}

\begin{proof}
    Proof similar to that of Lemma \ref{lem:discreteFiniteTransition}.
\end{proof}

\section{Pointwise convergence of the Kernel operators}

We start this section with the objective of proving that the class $g \cdot \mathcal{K}$ is a Glivenko-Cantelli class. This is the same as showing that the pointwise convergence of $P_{n,\mu}$ defined by \eqref{eq:OperatorPnDef} to $P_{\mu}$ defined by \eqref{eq:operatorOfKernel} happens almost surely. This is the content of Proposition \ref{prop:LawOfLargeNumber} which is proved using an estimate obtained in Lemma \ref{lem:boundOnIndivProbTerm}. The rest of the lemmas are also other probability estimates that follow the same proof structure as that of Lemma \ref{lem:boundOnIndivProbTerm}. These other lemmas will be useful when proving the compact convergence and numerical estimates for the rates of convergence in the next sections.

For the rest of this section we fix $M, a > 0$ and $\omega(\cdot) : [0,+\infty[ \rightarrow \mathbb{R}$ a non decreasing funcition such that $\lim_{\delta \rightarrow 0} \omega(\delta) = 0$. We assume that the kernel $k: \mathcal{X} \times \mathcal{X} \rightarrow \mathbb{R}$ satisfies the bound \eqref{eq:UpBoundKernel} and fix a measure $\mu \in \mathcal{U}_{a,\omega(\cdot), k}$ given by \eqref{eq:classOfMeasures}.

We start by recalling an important inequality from Probability Theory that will be essential for the rest of the work.

\begin{thm}\cite[Bernstein's inequality]{bernsteinProbBook}
\label{thm:bernseteinsInequality}
    Let $X_1,...,X_n$ be centered independent random variables such that $\sup_{i \in \{1,...,n\}} |X_i| \leq M$ and $\sup_{i \in \{1,...,n\}}\text{Var}(X_i) \leq V$ for some $M,V > 0$. Then for all $\epsilon > 0$
    \begin{equation}
        \mathbb{P}
        \left(
            \left\{
                \left|
                    \sum_{j=1}^n
                    \frac
                    {X_i(x)}
                    {n}
                \right|
                >
                \epsilon
            \right\}
        \right)
    \leq
        2
        \exp
        \left\{
            \frac{-n\epsilon^2}
            {2V + \frac{2}{3}\epsilon M}
        \right\}.
    \end{equation}
\end{thm}

Given $\delta > 0$, define the random variables given by
\begin{equation}
\label{eq:defOfRandomVariableDifference}
    Z_{j,\delta}(x)
:=
    \sup_{z \in B_\delta(x)}
    \left|
        k(x, X_j)
    -
        k(z, X_j)
    \right|
=
    k_\delta(x; X_j).
\end{equation}
Notice that
\begin{equation}
\label{eq:averageIsAnulli}
    \mathbb{E}[Z_{j,\delta}(x)]
=
    \|k_\delta(x;\cdot)\|_{L^1(\mathcal{X})}
\leq
    \omega(\delta).
\end{equation}
We can center these variables by
\begin{equation}
\label{eq:defOfRandVariable0mean}
    \overline{Z}_{j,\delta}(x)
:=
    Z_{j,\delta}(x)
-
    \mathbb{E}[Z_{j,\delta}(x)].
\end{equation}

\begin{lemma}
\label{lem:estimateForDifference}
    Given $\epsilon > 0$, $x \in \mathcal{X}$, we have for all $n \in \mathbb{N}$
    \begin{equation}
    \label{eq:bernsteinProb1}
        \mathbb{P}\left(
        \left\{
            \left|
                \sum_{j=1}^n \frac{\overline{Z}_{j,\delta}(x)}{n}
            \right|
        >
            \epsilon
        \right\}
    \right)
\leq
    2\exp\left\{
        \frac{-n\epsilon^2}{32M^2+\frac{8}{3}M\epsilon}
        \right\}.
    \end{equation}
    Also given $g \in B(\mathcal{X})$ we have
    \begin{equation}
        \mathbb{P}\Big(
            \left\{
            \left|
                P_{n,\mu}g(x)
            -
                P_{\mu}g(x)
            \right|
        >
            \epsilon
            \right\}
        \Big)
    \leq
        2\exp\left\{
            \frac{-n\epsilon^2}
            {8\|g\|_{B(\mathcal{X})}^2M^2 + \frac{4}{3}\|g\|_{B(\mathcal{X})} M\epsilon}
        \right\}
    \end{equation}
\end{lemma}
\begin{proof}
    By the definition of the random variables \eqref{eq:defOfRandVariable0mean} and equation \eqref{eq:UpBoundKernel}
    \begin{align}
        \left|
            \overline{Z}_{j,\delta}(x)
        \right|
    &\leq
        4M,\\
    \text{Var}(\overline{Z}_{j,\delta}(x))
    &\leq
        16M^2.
    \end{align}
    These are i.i.d distributions with zero mean, thus we can apply Bernstein's inequality \ref{thm:bernseteinsInequality} to conclude equation \eqref{eq:bernsteinProb1}.
    Consider the variables given by
    \begin{equation}
        Y_{j,\delta}(x)
    =
        k(x,X_j)g(X_j)
    -
        \mathbb{E}\left[
            k(x,X_j)g(X_j)
        \right].
    \end{equation}
    These variables are i.i.d distributions with zero mean, and satisfy
    \begin{align}
        |Y_{j,\delta}(x)|
    &\leq
        2\|g\|_{B(\mathcal{X})}M,\\
        \text{Var}
        (Y_{j,\delta}(x))
    &\leq
        4\|g\|_{B(\mathcal{X})}^2M^2,
    \end{align}
    and since the $X_i$ are distibuted uniformly with respect to $\mu$
    \begin{equation}
        \frac{1}{n}
        \sum_{i=1}^n
        Y_{j,\delta}(x)
    =
        P_{n,\mu}g(x) - Pg(x).
    \end{equation}
    We can thus apply Bernstein's inequality to conclude
    \begin{align}
        \mathbb{P}\left(
            \left\{
            \left|
                P_{n,\mu}g(x)
                -
                P_{\mu}g(x)
            \right|
            >
            \epsilon
            \right\}
        \right)
    &=
        \mathbb{P}\left(
            \left\{
            \left|
                \frac{1}{n}
                \sum_{i=1}^n
                Y_{j,\delta}(x)
            \right|
            >
            \epsilon
            \right\}
        \right)\\
    &\leq
        2\exp\left\{
            \frac{-n\epsilon^2}
            {8\|g\|_{B(\mathcal{X})}^2M^2 + \frac{4}{3}\|g\|_{B(\mathcal{X})} M\epsilon}
        \right\},
    \end{align}
    concluding the proof.
\end{proof}

The next lemma allows us to improve the pointwise bounds obtained by Bernstein's inequality into uniform bounds on the space $\mathcal{X}$.
\begin{lemma}
\label{lem:boundOnIndivProbTerm}
    Let $\epsilon > 0$ and $g \in B(\mathcal{X})$, and consider $\delta > 0$ such that $\|g\|_{B(\mathcal{X})}\omega(\delta) \leq \frac{1}{8}\epsilon$. Suppose $\mathcal{X}$ can be covered by $L_\delta = N(\mathcal{X}, \delta, d)$ balls of radius $\delta > 0$. Then for all $n \in \mathbb{N}$ we have
    \begin{align}
        &\mathbb{P}\left(
        \left\{
            \sup_{x \in \mathcal{X}}
            \left|
                P_{n,\mu}g(x)
            -
                P_{\mu}g(x)
            \right|
        >
            \epsilon
        \right\}
        \right)\\
    \leq&
        2L_\delta
        \left(
                \exp
                \left(\frac{-n\epsilon^2}{32M^2 + \frac{8}{3}M\epsilon}
            \right)
            +
        \exp\left(
                \frac{-n\epsilon^2}
                {32\|g\|_{B(\mathcal{X})}^2M^2 + \frac{8}{3} \|g\|_{B(\mathcal{X})} M \epsilon}
        \right)
        \right).
    \end{align}
    In particular if $\|g\|_{B(\mathcal{X})} \leq 1$ we have
    \begin{equation}
    \label{eq:simplifiedBound}
        \mathbb{P}\left(
        \left\{
            \sup_{x \in \mathcal{X}}
            \left|
                P_{n,\mu}g(x)
            -
                P_{\mu}g(x)
            \right|
        >
            \epsilon
        \right\}
        \right)
    \leq
        4L_\delta
                \exp
                \left(\frac{-n\epsilon^2}{32M^2 + \frac{8}{3}M\epsilon}
            \right).
    \end{equation}
\end{lemma}
\begin{proof}
    Consider a $\delta>0$ such as in the hypothesis. Let $\{B_{\delta}(x_i)\}_{i \in \{1,...,L_\delta\}}$ be a cover of $\mathcal{X}$. Given some $\delta_\epsilon>0$ (to be chosen later), we can use Lemma \ref{lem:estimateForDifference} and conclude that
    \begin{align}
        &\mathbb{P}\left(
        \left\{
            \sup_{x \in \mathcal{X}}
            \left|
                P_{n,\mu}g(x)
            -
                P_{\mu}g(x)
            \right|
        >
            \epsilon
        \right\}
        \right)\\
    \leq&
        \mathbb{P}\left(
        \left\{
            \sup_{x \in \mathcal{X}}
            \left|
                P_{n,\mu}g(x)
            -
                P_{\mu}g(x)
            \right|
        >
            \epsilon
        \right\}
        \cap
        \left\{
            \forall {i \in \{1,...,L_\delta\}}
            \left|
            \sum_{j=1}^n
            \frac{\overline{Z}_{j,\delta}(x_i)}{n}
            \right|
            <
            \delta_\epsilon
        \right\}
        \right)\\
    &\quad\quad
        +
        \sum_{i=1}^{L_\delta}
        \mathbb{P}\left(
            \left|
            \sum_{j=1}^n
            \frac{\overline{Z}_{j,\delta}(x_i)}{n}
            \right|
            \geq
            \delta_\epsilon
        \right)\\
    \leq&
        \mathbb{P}\left(
        \left\{
            \sup_{x \in \mathcal{X}}
            \left|
                P_{n,\mu}g(x)
            -
                P_{\mu}g(x)
            \right|
        >
            \epsilon
        \right\}
        \cap
        \left\{
            \forall {i \in \{1,...,L_\delta\}}
            \left|
            \sum_{j=1}^n
            \frac{\overline{Z}_{j,\delta}(x_i)}{n}
            \right|
            <
            \delta_\epsilon
        \right\}
        \right)\\
    &\quad\quad
        +
            2L_\delta
            \exp\left(\frac{-n\delta_\epsilon^2}{32M^2 + \frac{8}{3}M\delta_\epsilon}\right).
    \end{align}
    Now we show that for the right choice of $\delta_\epsilon > 0$ we have
    \begin{equation}
    \label{eq:desiredInclusionOnFinitePoints}
        \left\{
            \sup_{x \in \mathcal{X}}
            \left|
                P_{n,\mu}g(x)
            -
                P_{\mu}g(x)
            \right|
        >
            \epsilon
        \right\}
        \cap
        \left\{
            \forall {i \in \{1,...,L_\delta\}}
            \left|
            \sum_{j=1}^n
            \frac{\overline{Z}_{j,\delta}(x_i)}{n}
            \right|
        <
            \delta_\epsilon
        \right\}
    \subset
        \bigcup_{i=1}^{L_\delta}
        \left\{
            \left|
                P_{n,\mu}g(x_i)
            -
                P_{\mu}g(x_i)
            \right|
            >
            \frac{\epsilon}{2}
        \right\}.
    \end{equation}
    Consider $x \in \mathcal{X}$ such that
    \begin{align}
        \left|
            P_{n,\mu}g(x)
        -
            P_{\mu}g(x)
        \right|
    &>
        \epsilon,\\
         \forall {i \in \{1,...,L_\delta\}}
            \left|
            \sum_{j=1}^n
            \frac{\overline{Z}_{j,\delta}(x_i)}{n}
            \right|
    &<
        \delta_\epsilon.
    \end{align}
    Since $\{B_\delta(x_i)\}_{i \in \{1,...,L_\delta\}}$ is a cover of $\mathcal{X}$, there must exist some $i \in \{1,...,L_\delta\}$ satisfying $x \in B_\delta(x_i)$. With this point $x_i$ we use triangle inequality to obtain
    \begin{align}
    \label{eq:3wayIneq}
        \left|
            P_{n,\mu}g(x_i)
        -
            P_{n,\mu}g(x_i)
        \right|
    \geq
    -
        \left|
            P_{n,\mu}g(x_i)
        -
            P_{n,\mu}g(x)
        \right|
    +
        \left|
            P_{n,\mu}g(x)
        -
            P_{\mu}g(x)
        \right|
   - 
        \left|
            P_{\mu}g(x)
        -
            P_{\mu}g(x_i)
        \right|.
    \end{align}
    Since $x \in B_\delta(x_i)$ using equation \eqref{eq:maxDifferenceInBall}, we have
    \begin{align}
        \left|
            P_{\mu}g(x)
        -
            P_{\mu}g(x_i)
        \right|
    &\leq
        \int_\mathcal{X}
            \left|
                g(y)
            \right|
            \left|
                k(x,y)
            -
                k(x_i, y)
            \right|
            d\mu(y)\\
    &\leq
        \|
            g
        \|_{B(\mathcal{X})}
        \int_\mathcal{X}
            \sup_{z \in B_\delta(x_i)}
            \left|
                k(z,y)
            -
                k(x_i,y)
            \right|
            d\mu(y)\\
    &\leq
        \|g\|_{B(\mathcal{X})}
        \omega(\delta).
    \end{align}
    Similarly since $\sum_{j=1}^n\left|\frac{\overline{Z}_{j,\delta}(x_i)}{n}\right| < \delta_\epsilon$, we can use equation \eqref{eq:averageIsAnulli} to conclude
    \begin{align}
        \left|
            P_{n,\mu}g(x_i)
        -
            P_{n,\mu}g(x)
        \right|
    &\leq
        \frac{1}{n}
        \sum_{j=1}^n
        \left|
            g(X_j)
        \right|
        \left|
            k(x_i, X_j)
        -
            k(x, X_j)
        \right|\\
    &\leq
        \frac{1}{n}
        \|g\|_{B(\mathcal{X})}
        \sum_{j=1}^n
        Z_{j,\delta}(x_i)\\
    &=
        \frac{1}{n}
        \sum_{j=1}^n
        \|g\|_{B(\mathcal{X})}
        \left(
            \overline{Z}_{j,\delta}(x_i)
        +
            \mathbb{E}
            [
                Z_{j,\delta}(x_i)
            ]
        \right)\\
    &\leq
        \|g\|_{B(\mathcal{X})}\left(
            \delta_\epsilon
        +
            \omega(\delta)
        \right).
    \end{align}
    And so using \eqref{eq:3wayIneq} we have 
    \begin{equation}
        \left|  
            P_{n,\mu}g(x_i)
        -
            P_{\mu}g(x_i)
        \right|
    \geq
        \epsilon
    -
        \|g\|_{B(\mathcal{X})}
        \left(
            \delta_\epsilon
        +
            2\omega(\delta)
        \right).
    \end{equation}
    Choosing $\delta_\epsilon>0$ such that $\|g\|_{B(\mathcal{X})}\delta_\epsilon < \frac{\epsilon}{4}$ and the hypothesis $2\|g\|_{B(\mathcal{X})} \omega(\delta) \leq \frac{\epsilon}{4}$, we conclude the inclusion \eqref{eq:desiredInclusionOnFinitePoints}. We also choose $\delta_\epsilon < \epsilon$. From this we can apply Lemma \ref{lem:estimateForDifference} obtaining
    \begin{align}
        &\mathbb{P}
        \Bigg(
            \Bigg\{
            \sup_{x \in \mathcal{X}}
            \left|
                P_{n,\mu}g(x)
            -
                P_{\mu}g(x)
            \right|
        >
            \epsilon
        \Bigg\}
        \cap
        \Bigg\{
            \forall {i \in \{1,...,L_\delta\}}
            \left|
            \sum_{j=1}^n
            \frac{\overline{Z}_{j,\delta}(x_i)}{n}
            \right|
        <
            \delta_\epsilon
        \Bigg\}
        \Bigg)\\
    \leq&
        \sum_{i=1}^{L_\delta}
        \mathbb{P}
        \left(
            \left\{P_{n,\mu}g(x_i) - P_{\mu}g(x_i) > \frac{\epsilon}{2}
            \right\}
        \right)
    \leq
        2L_\delta
            \exp
            \left(
                \frac{-n\epsilon^2/4}
                {8\|g\|^2_{B(\mathcal{X})}M^2 + \frac{4}{3}\|g\|_{B(\mathcal{X})} M \epsilon/2}
            \right)\\
    =&
        2L_\delta
            \exp
            \left(
                \frac{-n\epsilon^2}
                {32\|g\|^2_{B(\mathcal{X})}M^2 + \frac{8}{3}\|g\|_{B(\mathcal{X})} M \epsilon}
            \right)
    \end{align}
    Thus we obtain the desired inequality
    \begin{align}
        &\mathbb{P}\left(
        \left\{
            \sup_{x \in \mathcal{X}}
            \left|
                P_{n,\mu}g(x)
            -
                P_{\mu}g(x)
            \right|
        >
            \epsilon
        \right\}
        \right)\\
    \leq&
        2
        L_\delta
        \left(
                \exp
                \left(\frac{-n\epsilon^2}{32M^2 + \frac{8}{3}M\epsilon}
            \right)
            +
        \exp\left(
                \frac{-n\epsilon^2}
                {32\|g\|_{B(\mathcal{X})}^2M^2 + \frac{8}{3} \|g\|_{B(\mathcal{X})}M \epsilon}
        \right)
        \right).
    \end{align}
\end{proof}

    \begin{prop}
    \label{prop:LawOfLargeNumber}
        Given $g \in B(\mathcal{X})$ we have
        \begin{equation}
            \mathbb{P}\left(
                \left\{
                    \limsup_n
                    \sup_{x \in \mathcal{X}}
                    \left|
                        P_{n, \mu}g(x)
                    -
                        P_{\mu}g(x)
                    \right|
                    =
                    0
                \right\}
            \right)
            =
            1,
        \end{equation}
        that is the class $g \cdot \mathcal{K}$ is a Glivenko Cantelli class. Equivalently $P_{n,\mu}\stackunder{$\longrightarrow$}{$\scriptsize{\text{p}}$} P_\mu$ almost surely.
    \end{prop}
    \begin{proof}
        Using continuity from above of the measure we have
        \begin{align}
    \label{eq:GlivenkoCantelliCondProof}
        \mathbb{P}\left(\left\{\limsup_n \sup_{x\in \mathcal{X}} |P_{n, \mu}g(x)-P_{\mu}g(x)| = 0\right\}\right)
   &=
        1
    -
        \mathbb{P}\left(\left\{\limsup_n \sup_{x\in \mathcal{X}} |P_{n, \mu}g(x)-P_{\mu}g(x)| >0\right\}\right)\\
    &=
        1
    -
        \lim_{\epsilon \rightarrow 0^+}\mathbb{P}\left(\left\{\limsup_n \sup_{x\in \mathcal{X}} |P_{n, \mu}g(x)-P_{\mu}g(x)| > \epsilon\right\}\right)\\
    &=
        1
    -
        \lim_{\epsilon \rightarrow 0^+}
        \lim_{n}
        \mathbb{P}\left(\bigcup_{j\geq n}\left\{\sup_{x\in \mathcal{X}} |P_{j,\mu}g(x)-P_{\mu}g(x)| > \epsilon\right\}\right)\\
    &\geq
        1
    -
        \lim_{\epsilon \rightarrow 0^+}
        \lim_{n}
        \sum_{j\geq n}
        \mathbb{P}\left(\left\{\sup_{x\in \mathcal{X}} |P_{j,\mu}g(x)-P_{\mu}g(x)| > \epsilon\right\}\right).
    \end{align}
    To study the term $\lim_{n}
        \sum_{j\geq n}
        \mathbb{P}\left(\left\{\sup_{x\in \mathcal{X}} |P_{j,\mu}g(x)-P_{\mu}g(x)| > \epsilon\right\}\right)$ choose $\delta > 0$ such that $\|g\|_{B(\mathcal{X})}\omega(\delta) \leq \frac{1}{8}\epsilon$ and using Lemma \ref{lem:boundOnIndivProbTerm} we can conclude that
        \begin{align}
        &\sum_{j\geq n}
        \mathbb{P}\left(\left\{\sup_{x\in \mathcal{X}} |P_{j,\mu}g(x)-P_{\mu}g(x)| > \epsilon\right\}\right)\\
        \leq&
        2L_\delta\sum_{j \geq n}\left(
                \exp
                \left(\frac{-j\epsilon^2}{32M^2 + \frac{8}{3}M\epsilon k}
            \right)
            +
        \exp\left(
                \frac{-j\epsilon^2}
                {32\|g\|^2_{B(\mathcal{X})}M^2 + \frac{8}{3}\|g\|_{B(\mathcal{X})} M \epsilon }
        \right)
        \right)
        \end{align}
        is a convergent series, and so its limit when $n \rightarrow \infty$ is zero. Thus we obtain
        \begin{equation}
            \lim_{n}
        \sum_{j\geq n}
        \mathbb{P}\left(\left\{\sup_{x\in \mathcal{X}} |P_{j,\mu}g(x)-P_{\mu}g(x)| > \epsilon\right\}\right) = 0,
        \end{equation}
        concluding the proof.
    \end{proof}

The following lemma will show a similar result to Lemma \ref{lem:boundOnIndivProbTerm}, but where $g$ is not only a function of $\mathcal{X}$ but is parametrized by $z \in \mathcal{X}$. In particular we will prove this for the kernel $k$. The proof follows exactly the same argument as that of Lemma \ref{lem:boundOnIndivProbTerm}. This is one of the lemmas that will be used when calculating the convergence rates of the spectral approximation.

\begin{lemma}
\label{lem:probBoundOnKernelWithKernel}
    Let $\epsilon > 0$, and consider $\delta > 0$ such that $\omega(\delta) \leq \frac{1}{16M}\epsilon$. Suppose $\mathcal{X}$ can be covered by $L_\delta = N(\mathcal{X},\delta, d)$ balls of radius $\delta > 0$. Then we have
    \begin{align}
        &\mathbb{P}
        \left(
            \left\{
                \sup_{x,z \in \mathcal{X}}
                    \left|
                        [P_{n, \mu} k(z,\cdot)](x)
                    -
                        [P_{\mu} k(z,\cdot)] (x)
                    \right|
                    >
                    \epsilon
            \right\}
        \right)\\
    \leq&
        2L_\delta
        \left(
        \exp\left(
            \frac{-n\epsilon^2}
            {32M^2 + \frac{8}{3}M\epsilon}
        \right)
        +
        L_\delta
                \exp\left(\frac{-n\epsilon^2}{32M^4 + \frac{8}{3}M^2\epsilon}\right)
        \right).
    \end{align}
    In particular we have
    \begin{equation}
        \mathbb{P}
        \left(
            \left\{
                \sup_{x,z \in \mathcal{X}}
                    \left|
                        [P_{n, \mu} k(z,\cdot)](x)
                    -
                        [P_{\mu} k(z,\cdot)] (x)
                    \right|
                    >
                    \epsilon
            \right\}
        \right)
    \leq
        4L_\delta^2
        \exp\left(\frac{-n\epsilon^2}{32M^4 + \frac{8}{3}M^2\epsilon}\right).
    \end{equation}
\end{lemma}
\begin{proof}
    Let $\{B_\delta(x_i)\}_{i \in \{1,...,L_\delta\}}$ be a cover of $\mathcal{X}$. Similarly to Lemma \ref{lem:boundOnIndivProbTerm}, we can show that
    \begin{align}
        &\mathbb{P}\left(
        \left\{
            \sup_{x,z \in \mathcal{X}}
            \left|
                [P_{n, \mu}[k(z,\cdot)](x)
            -
                [P_{\mu}k(z,\cdot)](x)
            \right|
        >
            \epsilon
        \right\}
        \right)\\
    \leq&
        \mathbb{P}\Bigg(
        \left\{
            \sup_{x,z \in \mathcal{X}}
            \left|
                [P_{n, \mu}k(z,\cdot)](x)
            -
                [P_{\mu}k(z,\cdot)](x)
            \right|
        >
            \epsilon
        \right\}
        \cap
        \left\{
            \forall {i \in \{1,...,L_\delta\}}
            \left|
            \sum_{j=1}^k
            \frac{\overline{Z}_{j,\delta}(x_i)}{k}
            \right|
            <
            \delta_\epsilon
        \right\}
        \Bigg)\\
    &\quad\quad
        +
            2L_\delta
            \exp\left(\frac{-n\delta_\epsilon^2}{32M^2 + \frac{8}{3}M\delta_\epsilon}\right).
    \end{align}

    We start by showing
    \begin{align}
    \label{eq:inclusionForKernelFamily}
        &\left\{
                \sup_{x,z \in \mathcal{X}}
                    \left|
                        [P_{n, \mu} k(z,\cdot)](x)
                    -
                        [P_{\mu} k(z,\cdot)] (x)
                    \right|
                    >
                    \epsilon
            \right\}
        \cap
        \left\{
            \forall {i \in \{1,...,L_\delta\}}
            \left|
            \sum_{j=1}^n
            \frac{\overline{Z}_{j,\delta}(x_i)}{n}
            \right|
            <
            \delta_\epsilon
        \right\}\\
        \subset&
            \bigcup_{i,j=1}^{L_\delta}
            \left\{
                    \left|
                        [P_{n, \mu} k(x_i,\cdot)](x_j)
                    -
                        [P_{\mu} k(x_i,\cdot)] (x_j)
                    \right|
                    >
                    \frac{\epsilon}{2}
            \right\}.
    \end{align}
    Consider $x \in \mathcal{X}$ such that
    \begin{align}
            \left|
                        [P_{n, \mu} k(z,\cdot)](x)
                    -
                        [P_{\mu} k(z,\cdot)] (x)
            \right|
        &>
            \epsilon,\\
            \forall {i \in \{1,...,L_\delta\}}
            \left|
            \sum_{j=1}^n
            \frac{\overline{Z}_{j,\delta}(x_i)}{n}
            \right|
        &<
            \delta_\epsilon
    \end{align}
    Since $\{B_\delta(x_i)\}_{i \in \{1,...,L_\delta\}}$ is a cover for $\mathcal{X}$, there exist $i,j \in \{1,...,L_\delta\}$ such that $x \in B_\delta(x_j)$ and $z \in B_\delta(x_i)$. Using triangle inequality we conclude 
    \begin{align}
        |
            [P_{n, \mu} k(x_i,\cdot)](x_j)
            -&
            [P_{\mu} k(x_i,\cdot)] (x_j)
        |
        >
        -\left|
            [P_{n, \mu} k(x_i,\cdot)](x_j) 
        -
            [P_{n, \mu} k(z, \cdot)](x_j)
        \right|\\
    &-
        \left|
            [P_{n, \mu} k(z,\cdot)](x_j) 
        -
            [P_{n, \mu} k(z, \cdot)] (x)
        \right|
    +
        \left|
            [P_{n, \mu} k(z, \cdot)] (x)
        -
            [P_{\mu} k(z, \cdot)] (x)
        \right|\\
    &-
        \left|
            [P_{\mu} k(z, \cdot)] (x)
        -
            [P_{\mu} k(z, \cdot)] (x_j)
        \right|
    -
        \left|
            [P_{\mu} k(z, \cdot)] (x_j)
        -
            [P_{\mu} k(x_i, \cdot)] (x_j)
        \right|.
    \end{align}
    Since $x \in B_\delta(x_j)$, using inequality \eqref{eq:maxDifferenceInBall}, we have
    \begin{align}
    \left|
        [P_{\mu} k(z, \cdot)] (x)
        -
        [P_{\mu} k(z, \cdot)] (x_j)
    \right|
&\leq
    \int_{\mathcal{X}}
        k(z,y)
        \left|
            k(x,y)
        -
            k(x_j,y)
        \right|
        d\mu(y)\\
&\leq
    \int_{\mathcal{X}}
        k(x,y)
        k_\delta(x_j; y)
        d\mu(y)\\
&\leq
    M\omega(\delta).
\end{align}
Similarly using the fact that $z \in B_\delta(x_i)$, then
\begin{equation}
    \left|
            [P_{\mu} k(z, \cdot)] (x_j)
        -
            [P_{\mu} k(x_i, \cdot)] (x_j)
    \right|
\leq
    M\omega(\delta).
\end{equation}
Due to $z \in B_\delta(x_i)$ and $\left|\sum_{l=1}^n\frac{\overline{Z}_{l,\delta}(x_i)}{n}\right| \leq \delta_\epsilon$, we have
\begin{align}
    \left|
            [P_{n, \mu} k(x_i,\cdot)](x_j) 
        -
            [P_{n, \mu} k(z, \cdot)](x_j)
    \right|
&\leq
    \left|
        \frac{1}{n}
        \sum_{l=1}^n
            k(x_j, X_l)
            k(x_i, X_l)
        -
            k(x_j, X_l)
            k(z, X_l)
    \right|\\
&\leq
        \frac{1}{n}
        \sum_{l=1}^n
        \left|
            k(x_j, X_l)
        \right|
        \left|
            k(x_i, X_l)
        -
            k(z, X_l)
        \right|\\
&\leq
    \frac{1}{n}M
        \sum_{l=1}^n
        k_\delta(x_i; X_l)\\
&=
    \frac{1}{n}M
    \sum_{l=1}^n
    \left(
        \overline{Z}_{l,\delta}(x_i)
    +
        \mathbb{E}[Z_{l,\delta}(x_i)]
    \right)\\
&\leq
    M\left(
        \delta_\epsilon
    +
        \omega(\delta)
    \right).
\end{align}
Using the same argument as above,
\begin{equation}
    \left|
            [P_{n, \mu} k(z,\cdot)](x_j) 
        -
            [P_{n, \mu} k(z, \cdot)] (x)
        \right|
\leq
    M(\delta_\epsilon + \omega(\delta)).
\end{equation}
From these equations we conclude
\begin{equation}
    |
        [P_{n, \mu} k(x_i,\cdot)](x_j)
        -
        [P_{\mu} k(x_i,\cdot)] (x_j)
    |
\geq
    \epsilon -4M\omega(\delta) - 2M\delta_\epsilon.
\end{equation}
If we choose $\delta_\epsilon,\delta > 0$ such that $\omega(\delta) \leq \frac{1}{16M}\epsilon$, $\delta_\epsilon < \frac{1}{8M}\epsilon$ and $\delta_\epsilon < \epsilon$, we conclude
\begin{equation}
    |
        [P_{n, \mu} k(x_i,\cdot)](x_j)
        -
        [P_{\mu} k(x_i,\cdot)] (x_j)
    |
>
    \frac{\epsilon}{2},
\end{equation}
showing equation \eqref{eq:inclusionForKernelFamily}. Now to conclude the proof, we can use Lemma \ref{lem:estimateForDifference} and inequality \eqref{eq:UpBoundKernel} to show
\begin{align}
    &\mathbb{P}\left(
        \bigcup_{i,j=1}^{L_\delta}
            \left\{
                    \left|
                        [P_{n, \mu} k(x_i,\cdot)](x_j)
                    -
                        [P_{\mu} k(x_i,\cdot)] (x_j)
                    \right|
                    >
                    \frac{\epsilon}{2}
            \right\}
    \right)\\
\leq&
    \sum_{i=1}^{L_\delta}
    2L_\delta
    \exp\left(
        \frac{-n\epsilon^2/4}
        {8\|k(x_i,\cdot)\|_{B(\mathcal{X})}^2M^2 + \frac{4}{3}\|k(x_i,\cdot)\|_{B(\mathcal{X})}M\epsilon/2}
    \right)\\
\leq&
    2L_\delta^2
    \exp\left(
        \frac{-n\epsilon^2}
        {32M^4 + \frac{8}{3}M^2\epsilon}
    \right),
\end{align}
finishing the first inequality. For the second use the fact that $M\geq 1$.
\end{proof}

Now we are going to apply the previous lemmas to the kernel $h_\mu$ defined by \eqref{eq:definitionOfH}. We do this  by showing that the kernel $h_\mu$ satisfies similar properties to $k$, the ones used in the previous proofs, and then we will apply the results directly.
\begin{lemma}
\label{lem:estimatesForH}
    Define the function given by
    \begin{equation}
    \label{eq:defOfTildeO}
        \tilde{\omega}(\delta) 
    = 
        \frac{\omega(\delta)}{2}
        \left(
            \frac{1}{a}
        +
            \frac{2M}{a^2}
        \right).
    \end{equation}
    Let $h_\mu$ be given by \eqref{eq:definitionOfH}. Then we have for all $x \in \mathcal{X}$
    \begin{equation}
    \label{eq:boundOnH1}
        \|
            (h_\mu)_\delta(x;\cdot)
        \|_{L^1(\mathcal{X})}
        \leq
        \tilde{\omega}(\delta),
    \end{equation}
    \begin{equation}
    \label{eq:boundOnH2}
        \sup_{x,y \in \mathcal{X}}|h_{\mu}(x,y)| 
    \leq  
        \frac{M}{a},
    \end{equation}
    \begin{equation}
    \label{eq:boundOnH3}
        \frac{a}{M}
    \leq
        \inf_{x \in \mathcal{X}}
            \|
                h_{\mu}(x,\cdot)
            \|_{L^1(\mathcal{X})}.
    \end{equation}
 That is, if $\mu \in \mathcal{U}_{a, \omega(\cdot), k}$ then $\mu \in \mathcal{U}_{\frac{a}{M}, \tilde{\omega}(\cdot), h_\mu}$.
\end{lemma}
\begin{proof}
    For equation \eqref{eq:boundOnH1} we use inequalities \eqref{eq:UpBoundKernel}, \eqref{eq:UpAndLowBound} and \eqref{eq:maxDifferenceInBall} to conclude for all $x \in \mathcal{X}$
    \begin{align}
        2
        \|(h_\mu)_\delta(x;\cdot)\|_{L^1(\mathcal{X})}
        &=
        2
        \int_{\mathcal{X}}
            \sup_{y \in B_\delta(x)}
            \left|
                h_{\mu}(x,z)
            -
                h_{\mu}(y,z)
            \right|
            d\mu(z)\\
    &\leq
        \int_\mathcal{X}
            \sup_{y \in B_\delta(x)}
            \left|
                \left(
                    \frac{k(x,z)}
                    {d_{\mu}(x)}
                +
                    \frac{k(x,z)}
                    {d_{\mu}(z)}
                \right)
            -
                \left(
                    \frac{k(y,z)}
                    {d_{\mu}(y)}
                +
                    \frac{k(y,z)}
                    {d_{\mu}(z)}
                \right)
            \right|
            d\mu(z)\\
    &\leq
        \int_\mathcal{X}
            \sup_{y \in B_\delta(x)}
            \left(
            \left|
                \frac{k(x,z) - k(y,z)}
                {d_{\mu}(z)}
            \right|
            +
            \left|
                    \frac{k(x,z)}
                    {d_{\mu}(x)}
            -
                    \frac{k(y,z)}
                    {d_{\mu}(y)}
            \right|
            \right)
            d\mu(z)\\
    &\leq
        \int_\mathcal{X}
            \sup_{y \in B_\delta(x)}
            \left(
            \left|
                \frac{k(x,z) - k(y,z)}
                {d_{\mu}(z)}
            \right|
            +
            \left|
                    \frac{k(x,z)d_{\mu}(y)}
                    {d_{\mu}(x)d_{\mu}(y)}
            -
                    \frac{k(y,z)d_{\mu}(x)}
                    {d_{\mu}(y)d_{\mu}(x)}
            \right|
            \right)
            d\mu(z)\\
    &\leq
        \frac{\omega(\delta)}
        {a}
    +
    \int_\mathcal{X}
        \sup_{y \in B_\delta(x)}
            \left|
                    \frac{k(x,z)(d_{\mu}(y)-d_{\mu}(x))}
                    {d_{\mu}(x)d_{\mu}(y)}
                -
                    \frac{(k(x,z)-k(y,z))d_{\mu}(x)}
                    {d_{\mu}(y)d_{\mu}(x)}
            \right|
            d\mu(z)\\
    &\leq
        \frac{\omega(\delta)}
        {a}
    +
        \frac{M}{a^2}
        \int_\mathcal{X}
            \sup_{y \in B_\delta(x)}
                \left|
                    d_{\mu}(y)
                -
                    d_{\mu}(x)
                \right|
            +
                \left|
                    k(x,z)
                -
                    k(y,z)
                \right|
            d\mu(z)\\
    &\leq
        \omega(\delta)
        \left(
            \frac{1}{a}
        +
            \frac{2M}{a^2}
        \right).
    \end{align}
    The bound \eqref{eq:boundOnH2} follows directly from the definitions. For inequality \eqref{eq:boundOnH3}, given $x \in \mathcal{X}$ we have
    \begin{align}
        2
        \|
            h_{\mu}(x,\cdot)
        \|_{L^1(\mathcal{X})}
    &=
        \int_\mathcal{X}
        \left|
            \frac{k(x,z)}
            {d_{\mu}(x)}
        +
            \frac{k(x,z)}
            {d_\mu(z)}
        \right|
        d\mu(z)
    \geq
        \frac{2}{M}
        \int_\mathcal{X}
        k(x,z)
        d\mu(z)
    \geq
        \frac{2a}{M}.
    \end{align}
\end{proof}

\begin{corollary}
\label{cor:CrossedTermEstimate}
    Let $\epsilon >0$, and consider $\delta > 0$ such that $\tilde{\omega}(\delta) \leq \frac{a}{16M}\epsilon$. Suppose $\mathcal{X}$ can be covered by $L_\delta = N(\mathcal{X},\delta, d)$ balls of radius $\delta > 0$. Then we have
    \begin{equation}
        \mathbb{P}\left(
            \left\{
                \sup_{x,z \in \mathcal{X}}
                \left|
                    [T_{n,\mu}h_\mu(z,\cdot)](x)
                -
                    [T_\mu h_\mu(z,\cdot)](x)
                \right|
            \right\}
        \right)
    \leq
        4L_\delta^2
        \exp
        \left(
            \frac{-na^2\epsilon^2}
            {32M^4/a^2 + \frac{8}{3}M^2\epsilon}
        \right).
    \end{equation}
\end{corollary}
\begin{proof}
    Using Lemma \ref{lem:estimatesForH} we have that $\mu \in \mathcal{U}_{\frac{a}{M}, \tilde{\omega}(\cdot), h_\mu}$, and
    \begin{equation}
        \sup_{x,y \in \mathcal{X}} |h_\mu(x,y)| 
    \leq
        \frac{M}{a}.
    \end{equation}
    With this we can apply Lemma \ref{lem:probBoundOnKernelWithKernel} for the kernel $h_\mu$ to conclude that
    \begin{align}
        \mathbb{P}\left(
            \left\{
                \sup_{x,z \in \mathcal{X}}
                \left|
                    [T_{n,\mu}h_\mu(z,\cdot)](x)
                -
                    [T_\mu h_\mu(z,\cdot)](x)
                \right|
            \right\}
        \right)
    \leq&
        4L_\delta^2
        \exp
        \left(
            \frac{-na^2\epsilon^2}
            {32M^4/a^2 + \frac{8}{3}M^2\epsilon}
        \right)\\
    \leq&
        4L_\delta^2
        \exp
        \left(
            -\frac{n\epsilon^2}
            {32(M/a)^4 + \frac{8}{3}(M/a)^2\epsilon}
        \right)\\
    \leq&
        4L_\delta^2
        \exp
        \left(
            -\frac{na^2\epsilon^2}
            {32M^4/a^2 + \frac{8}{3}M^2\epsilon}
        \right).
    \end{align}
\end{proof}

    \begin{remark}
    \label{rem:averageKernelRemark}
     Since $\mu \in \mathcal{U}_{a, \omega(\cdot), k}$ then by Lemma \ref{lem:estimatesForH} $\mu \in \mathcal{U}_{\frac{1}{Ma}, \tilde{\omega}(\cdot), h_\mu}$. We can thus use Proposition \ref{prop:LawOfLargeNumber} to conclude that given $g \in B(\mathcal{X})$, then $g\cdot \mathcal{H}$ is also a Glivenko-Cantelli class. By Lemma \ref{lem:estimatesForH}, $h_\mu$ satisfies the same conditions as $k$.Thus
    \begin{equation}
        \sup_{x \in \mathcal{X}}
        \left|
            \int_{\mathcal{X}}
            h_{\mu}(x,y)
            g(y)
            \left(
                d\mu_n(y)
                -
                d\mu(y)
            \right)
        \right|
    \end{equation}
    converges to zero almost surely using Proposition \ref{prop:LawOfLargeNumber}.
\end{remark}

\begin{lemma}
\label{lem:boundForDifferenceOfDifferences}
    Given $\delta > 0$, $x \in \mathcal{X}$ let $z \in B_\delta(x)$. Then
    \begin{equation}
        \left|
            k_\delta(x;y)
        -
            k_\delta(z;y)
        \right|
    \leq
        3k_{2\delta}(x;y).
    \end{equation}      
\end{lemma}
\begin{proof}
    Consider $x \in \mathcal{X}$, $z \in B_\delta(x)$, $w \in B_\delta(x)$ and $v \in B_\delta(z)$. By triangle inequality  we have $v \in B_{2\delta}(x)$. Then we have
    \begin{align}
    &
        \left|
        (
            k(x,y)
        -
            k(w,y)
        )
    -
        (
            k(z,y)
        -
            k(v,y)
        )
        \right|\\
    \leq&
        \left|
            k(x,y)
        -
            k(z,y)
        \right|
    +
        \left|
            k(x,y)
        -
            k(z,y)
        \right|
    +
        \left|
        k(v,y)
    -
        k(x,y)
        \right|\\
    \leq&
        3k_{2\delta}(x,y).
    \end{align}
    Taking the supremum over $w \in B_\delta(x)$ and $v \in B_\delta(z)$ concludes the proof.
\end{proof}
Now we are going to make an estimate that is true almost surely, and will be useful in the next section for proving the collective compact convergence. The proof follows again the same structure of the previous results, namely Lemma \ref{lem:boundOnIndivProbTerm}.

\begin{lemma}
    \label{lem:differenceOfKernelEstimate}
        Given $\delta>0$ we have
        \begin{equation}
            \mathbb{P}
            \left(
                \left\{
                    \limsup_{n\rightarrow \infty}
                    \sup_{x \in \mathcal{X}}
                    \left|
                    \int_\mathcal{X}
                    k_\delta(x;y)
                        \left(
                            d\mu_n(y)
                            -
                            d\mu(y)
                        \right)
                    \right|
                \leq
                    10\omega(2\delta)
                \right\}
            \right)
            =
            1.
        \end{equation}
\end{lemma}
\begin{proof}
    As in the proof of Proposition \ref{prop:LawOfLargeNumber} we use continuity of the measure to obtain
    \begin{align}
        &\mathbb{P}
            \left(
                \left\{
                    \limsup_n
                    \sup_{x \in \mathcal{X}}
                    \left|
                    \int_\mathcal{X}
                    k_\delta(x;y)
                        \left(
                            d\mu_n(y)
                            -
                            d\mu(y)
                        \right)
                    \right|
                \leq
                    10\omega(2\delta)
                \right\}
            \right)\\
        =&
            1
        -
            \mathbb{P}
            \left(
                \left\{
                    \limsup_n
                    \sup_{x \in \mathcal{X}}
                    \left|
                    \int_\mathcal{X}
                    k_\delta(x;y)
                        \left(
                            d\mu_n(y)
                            -
                            d\mu(y)
                        \right)
                    \right|
                >
                    10\omega(2\delta)
                \right\}
            \right)\\
        =&
            1
        -
            \lim_n
            \mathbb{P}
            \left(
                \bigcup_{j \geq n}
                \left\{
                    \sup_{x \in \mathcal{X}}
                    \left|
                    \int_\mathcal{X}
                    k_\delta(x;y)
                        \left(
                            d\mu_j(y)
                            -
                            d\mu(y)
                        \right)
                    \right|
                >
                    10\omega(2\delta)
                \right\}
            \right)\\
        \geq&
            1
            -
            \lim_n
            \sum_{j \geq n}
            \mathbb{P}
            \left(
                \left\{
                    \sup_{x \in \mathcal{X}}
                    \left|
                    \int_\mathcal{X}
                    k_\delta(x;y)
                        \left(
                            d\mu_j(y)
                            -
                            d\mu(y)
                        \right)
                    \right|
                >
                    10\omega(2\delta)
                \right\}
            \right).
    \end{align}
    Again similarly to Proposition \eqref{prop:LawOfLargeNumber}, we show that 
    \begin{equation}
        \lim_n
            \sum_{j \geq n}
            \mathbb{P}
            \left(
                \left\{
                    \sup_{x \in \mathcal{X}}
                    \left|
                    \int_\mathcal{X}
                    k_\delta(x;y)
                        \left(
                        d\mu_j(y)
                        -
                        d\mu(y)
                        \right)
                    \right|
                >
                    10\omega(2\delta)
                \right\}
            \right)
        =
            0.
    \end{equation}
    Given $\epsilon > 0$ let $\delta_\epsilon > 0$ (to be chosen later), and consider $\{B_{\delta}(x_i)\}_{i \in \{1,...,L_{\delta}\}}$ a cover of $\mathcal{X}$. Given $\delta_\epsilon > 0$, using Lemma \ref{lem:estimateForDifference} we have
    \begin{align}
        &
        \lim_n
            \sum_{j \geq n}
            \mathbb{P}
            \left(
                \left\{
                    \sup_{x \in \mathcal{X}}
                    \left|
                    \int_\mathcal{X}
                    k_\delta(x;y)
                        \left(
                        d\mu_j(y)
                        -
                        d\mu(y)
                        \right)
                    \right|
                >
                    10\omega(2\delta)
                \right\}
            \right)\\
        =&
        \lim_n
            \sum_{j \geq n}
            \mathbb{P}
            \left(
                \left\{
                    \sup_{x \in \mathcal{X}}
                    \left|
                    \int_\mathcal{X}
                    k_\delta(x;y)
                        \left(
                        d\mu_j(y)
                        -
                        d\mu(y)
                        \right)
                    \right|
                >
                    10\omega(2\delta)
                \right\}
                \cap
                \left\{
                \forall_{i\in\{1,...,L_\delta\}}
                \left|
                    \sum_{l = 1}^j \frac{\overline{Z}_{l, 2\delta}(x_i)}{j}
                \right|
            \leq
        \delta_\epsilon
        \right\}
            \right)\\
        &\quad\quad+
        \sum_{j\geq n}
            \sum_{i=1}^{L_\delta}
            \mathbb{P}
            \left(
                \left|
                    \sum_{j = 1}^j \frac{\overline{Z}_{j, 2\delta}(x_i)}{j}
                \right|
            >
            \delta_\epsilon
            \right)\\
        \leq&
            \lim_n
            \sum_{j \geq n}
            \mathbb{P}
            \left(
                \left\{
                    \sup_{x \in \mathcal{X}}
                    \left|
                    \int_\mathcal{X}
                    k_\delta(x;y)
                        \left(
                        d\mu_j(y)
                        -
                        d\mu(y)
                        \right)
                    \right|
                >
                    10\omega(2\delta)
                \right\}
                \cap
                \left\{
                \forall_{i\in\{1,...,L_\delta\}}
                \left|
                    \sum_{l = 1}^j \frac{\overline{Z}_{l, 2\delta}(x_i)}{j}
                \right|
            \leq
        \delta_\epsilon
        \right\}
            \right)\\
        &\quad\quad+
        \lim_n
        \sum_{j\geq n}
            \sum_{i=1}^{L_\delta}
            2\exp
            \left\{
                \frac{
                -j\delta_\epsilon^2
                }
                {32M^2 + \frac{8}{3}M\delta_\epsilon}
            \right\}\\
        =&
            \lim_n
            \sum_{j \geq n}
            \mathbb{P}
            \left(
                \left\{
                    \sup_{x \in \mathcal{X}}
                    \left|
                    \int_\mathcal{X}
                    k_\delta(x;y)
                        \left(
                        d\mu_j(y)
                        -
                        d\mu(y)
                        \right)
                    \right|
                >
                    10\omega(2\delta)
                \right\}
                \cap
                \left\{
                \forall_{i\in\{1,...,L_\delta\}}
                \left|
                    \sum_{l = 1}^j \frac{\overline{Z}_{l, 2\delta}(x_i)}{j}
                \right|
            \leq
        \delta_\epsilon
        \right\}
            \right).
    \end{align}
    Now we prove that, with $\delta_\epsilon < \omega(2\delta)$ we have
    \begin{align}
\label{eq:intersectionOfConditions2}
        \Big\{
            \sup_{x \in \mathcal{X}}
                    &
                    \left|
                    \int_\mathcal{X}
                    k_\delta(x;y)
                        \left(
                        d\mu_j(y)
                        -
                        d\mu(y)
                        \right)
                    \right|
            >
                10\omega(2\delta)
        \Big\}
    \cap
        \left\{
            \forall_{i\in\{1,...,L_\delta\}}
            \left|
                \sum_{l = 1}^j \frac{\overline{Z}_{l, \delta}(x_i)}{j}
            \right|
        \leq
    \delta_\epsilon
        \right\}\\
    &\subset
        \left\{
            \max_{i \in \{1,...,L_\delta\}}
                \left|
                    \int_\mathcal{X}
                    k_\delta(x_i;y)
                        \left(
                            d\mu_j(y)
                        -
                            d\mu(y)
                        \right)
                    \right|
            >
                4\omega(2\delta)
        \right\}.
        \label{eq:inclusionInFiniteSet2}
    \end{align}
    Indeed given $x \in \mathcal{X}$ such that 
    \begin{equation}
        \left|
                    \int_\mathcal{X}
                    k_\delta(x;y)
                        \left(
                        d\mu_j(y)
                        -
                        d\mu(y)
                        \right)
                    \right|
            >
                10\omega(2\delta)
    \end{equation}
    there exists $i \in \{1,...,L_{\delta}\}$ such that $x \in B_{\delta}(x_i)$. Then by triangle inequality, equation \eqref{eq:maxDifferenceInBall} and Lemma \ref{lem:boundForDifferenceOfDifferences} we conclude
    \begin{align}
        \left|
                    \int_\mathcal{X}
                    k_\delta(x_i;y)
                    \left(
                        d\mu_j(y)
                    -
                        d\mu(y)
                    \right)
        \right|
    &\geq
        -
                \left|
                    \int_\mathcal{X}
                    \left(
                    k_\delta(x;y)
                    -
                    k_\delta(x_i;y)
                    \right)
                        \left(
                        d\mu_j(y)
                    -
                        d\mu(y)
                        \right)
                \right|\\
    &\quad\quad+
        \left|
            \int_\mathcal{X}
                k_\delta(x;y)\left(d\mu_j(y)-d\mu(y)\right)
        \right|\\
    &\geq
        -
            \int_\mathcal{X}
                    \left|
                    k_\delta(x;y)
                    -
                    k_\delta(x_i;y)
                    \right|
                        d\mu_j(y)\\
    &\quad\quad-
            \int_\mathcal{X}
                \left(
                    k_\delta(x;y)
                    +
                    k_\delta(x_i;y)
                \right)
                d\mu(y)
        +10\omega(2\delta)\\
    &\geq
        -
            3\int_\mathcal{X}
                k_{2\delta}(x_i;y)
                        d\mu_j(y) 
        -2\omega(2\delta)
        +10\omega(2\delta)\\
    &\geq
        -3\left|
            \int_\mathcal{X} k_{2\delta}(x_i;y)\left(d\mu_j(y)-d\mu(y)\right)
            \right|
        -
        \int_\mathcal{X} k_{2\delta}(x_i;y)d\mu(y)
        +
        8\omega(2\delta)\\
    &\geq
        -3\omega(2\delta)
        -\omega(2\delta)
        +8\omega(2\delta)
    \geq
        4\omega(2\delta).
    \end{align}
    Thus applying Lemma \ref{lem:estimateForDifference} we have
    \begin{align}
        &\lim_n
            \sum_{j \geq n}
            \mathbb{P}
            \left(
                \left\{
                    \sup_{x \in \mathcal{X}}
                    \left|
                    \int_\mathcal{X}
                    k_\delta(x;y)
                        \left(
                        d\mu_j(y)
                        -
                        d\mu(y)
                        \right)
                    \right|
                >
                    10\omega(2\delta)
                \right\}
                \cap
                \left\{
                \forall_{i\in\{1,...,L_\delta\}}
                \left|
                    \sum_{l = 1}^j \frac{\overline{Z}_{l, 2\delta}(x_i)}{j}
                \right|
            \leq
        \delta_\epsilon
        \right\}
            \right)\\
        \leq&
        \lim_n
            \sum_{j \geq n}
            \mathbb{P}
            \left(
                \left\{
            \max_{i \in \{1,...,L_\delta\}}
                \left|
                    \int_\mathcal{X}
                    k_\delta(x_i;y)
                        \left(
                            d\mu_j(y)
                        -
                            d\mu(y)
                        \right)
                    \right|
            >
                4\omega(2\delta)
        \right\}
            \right)\\
    \leq&
        \lim_n
            \sum_{j \geq n}
            \sum_{i=1}^{L_\delta}
            \mathbb{P}
            \left(
                \left\{
                \left|
                    \frac{1}{k}
                    \sum_{l=1}^j
                    \overline{Z}_{l,\delta}(x_i)
                    \right|
            >
                4\omega(2\delta)
        \right\}
            \right)\\
    \leq&\lim_n
            \sum_{j \geq n}
            \sum_{i=1}^{L_\delta}
        2\exp
            \left\{
                \frac{
                -j(4 \omega(2\delta))^2
                }
                {32M^2 + \frac{4}{3}M(4\omega(2\delta))}
            \right\}
    =
        0.
    \end{align}
    This finishes the proof.
\end{proof}

\section{Convergence Of Operators}

This section is dedicated to proving the conditions of Theorem \ref{thm:funcAnalAbstractSpectralConvergence} and \ref{thm:quantifiedSpectralConvergence}, for the operators $\hat{T}_{n, \mu}$ defined by \eqref{eq:THatDef}, so we can obtain the almost sure convergence of the eigenvectors of the operators $U_{n, \mu}$ and $U_{n,\mu}'$. For this we start by proving that almost surely the operator $\hat{T}_{n,\mu}$ converges to $\hat{T}_\mu$ pointwise in Proposition \ref{prop:pointWiseConvergence}. We then proceed to prove the compact collective convergence of $\hat{T}_{n,\mu}$ to $\hat{T}_{\mu}$ happens almost surely. With this we will be able to conclude the compact convergence of both the operators $U_{n, \mu}$ and $U_{n,\mu}'$ to $U_\mu$ and $U_{\mu}'$ respectively. We finish the section by proving Theorems \ref{thm:llneigenfunctions1} and \ref{thm:llneigenfunctions2}.

As in the previous section we fix $M, a > 0$ and $\omega(\cdot) : [0,+\infty[ \rightarrow \mathbb{R}$ a non decreasing funcition such that $\lim_{\delta \rightarrow 0} \omega(\delta) = 0$. We assume that the kernel $k: \mathcal{X} \times \mathcal{X} \rightarrow \mathbb{R}$ satisfies the bound \eqref{eq:UpBoundKernel} and fix a measure $\mu \in \mathcal{U}_{a,\omega(\cdot), k}$ given by \eqref{eq:classOfMeasures}.

\begin{prop}
\label{prop:pointWiseConvergence}
    We have $\hat{T}_{n, \mu} \stackunder{$\longrightarrow$}{$\scriptsize{\text{p}}$} T_\mu$ almost surely.
\end{prop}
\begin{proof}
    Given $f \in B(\mathcal{X})$ then
    \begin{equation}
        \|
            \hat{T}_{n, \mu}f
        -
            T_\mu f
        \|_{B(\mathcal{X})}
    \leq
        \|
            \hat{T}_{n, \mu}f
        -
            T_{n, \mu}f
        \|_{B(\mathcal{X})}
    +
        \|
            T_{n, \mu}f
        -
            T_\mu f
        \|_{B(\mathcal{X})}.
    \end{equation}
    The second term on the right hand side of the equation can be written as
    \begin{align}
        \|
            T_{n, \mu}f
        -
            T_\mu f
        \|_{B(\mathcal{X})}
    &=
        \sup_{x \in \mathcal{X}}
        \left|
            \int_{\mathcal{X}}
            h_{\mu}(x,y)
            f(y)
            \left(
                d\mu_n(y)
                -
                d\mu(y)
            \right)
        \right|
    \end{align}
    which by Remark \ref{rem:averageKernelRemark} we know converges almost surely to zero.

    Now we notice that since there exists $a > 0$ such that $0< a < \inf_{x \in \mathcal{X}}d_{\mu}(x)$, applying Proposition \ref{prop:LawOfLargeNumber} with $g = 1$, with probability $1$, there exists $N$ such that for all $n > N$ then $a < \inf_{x \in \mathcal{X}} d_{n,\mu}(x)$. Thus, for $n> N$ we have
    \begin{align}
        \|
            \hat{T}_{n, \mu}f
        -
            T_{n, \mu}f
        \|_{B(\mathcal{X})}
    &\leq
        \sup_{x\in \mathcal{X}}
        \left|
            \int_{\mathcal{X}}
            \frac{k(x,y)}{2}
            f(y)
            \left(
                \frac{1}{d_{\mu,n}(x)}
                +
                \frac{1}{d_{\mu,n}(y)}
                -
                \frac{1}{d_\mu(x)}
                -
                \frac{1}{d_\mu(y)}
            \right)
            d\mu(y)
        \right|\\
    &\leq
        \frac{M\|f\|_{B(\mathcal{X})}
       }{2}
        \sup_{x,y\in \mathcal{X}}
        \left|
            \frac{1}{d_{\mu,n}(x)}
            -
                \frac{1}{d_\mu(x)}
                +
                \frac{1}{d_{\mu,n}(y)}
                -
                \frac{1}{d_\mu(y)}
        \right|\\
    &\leq
        \frac{M\|f\|_{B(\mathcal{X})}
        }{2}
        \sup_{x,y\in \mathcal{X}}
        \left|
            \frac{d_\mu(x)-d_{\mu,n}(x)}
            {d_\mu(x)d_{\mu,n}(x)}
            +
            \frac{d_\mu(y)-d_{\mu,n}(y)}
            {d_\mu(y)d_{\mu,n}(y)}
        \right|\\
    &\leq
        \frac{M\|f\|_{B(\mathcal{X})}
        }
        {2a^2}
        \sup_{x,y\in \mathcal{X}}
        \left|
            d_{\mu}(x) - d_{n,\mu}(x)
        \right|
        +
        \left|
            d_{\mu}(y) - d_{n,\mu}(y)
        \right|.
    \end{align}
    Again applying Proposition \ref{prop:LawOfLargeNumber} we know that almost surely $\sup_{x\in \mathcal{X}}|d_{\mu}(x) - d_{n,\mu}(x)| \rightarrow 0$ and so the proof follows.
\end{proof}

Now we prove a "pseudo" equicontinuity of a sequence $\hat{T}_{n,\mu} f_n$ when $f_n$ is bounded in $B(\mathcal{X})$. This will be useful to us because to show the collective compact convergence of $\hat{T}_n$ we will resort also to a proof similar to that of the  Arzelà–Ascoli Theorem, even though the functions $\hat{T}_{n,\mu}f_n$ are not continuous. In some sense this result is related (but not completely), to the fact that the operators $\hat{T}_{n,\mu}$, (which may have discontinuous functions in its image) converges to the operator $T_\mu$ that by Proposition \ref{prop:imageInContinuousFunction} has its image in the continuous functions $C(\mathcal{X})$. In particular $T_\mu$ is a compact operator to the space of continuous functions, since the image of the unit ball of $B(\mathcal{X})$ to $C(\mathcal{X})$ is an equicontinuous set (see Proposition \ref{prop:imageInContinuousFunction}).

\begin{lemma}
\label{lem:equicontinuity}
    Given a sequence $f_n \in B(\mathcal{X})$ uniformly bounded, almost surely, for all $\epsilon > 0$, there exist $\delta > 0$ and $N \in \mathbb{N}$ such that if $n > N$ and $d(x,x') < \delta$ then
    \begin{equation}
        |\hat{T}_{n,\mu}f_n(x) - \hat{T}_{n,\mu} f_n(x')| < \epsilon,
    \end{equation}
    or equivalently
    \begin{equation}
        \mathbb{P}
        \left(
            \left\{
                \lim_{\delta \rightarrow 0, n \rightarrow \infty}
                \sup_{x\in \mathcal{X}}
                \sup_{x' \in B_\delta(x)}
                \left|
                    \hat{T}_{n,\mu}f_n(x)
                -
                    \hat{T}_{n,\mu}f_n(x')
                \right|
                =
                0
            \right\}
        \right)
        =
        1
    \end{equation}
\end{lemma}
\begin{proof}
    Let $f_n \in B(\mathcal{X})$ be a sequence such that \begin{equation}
    \label{eq:boundOnSequenceLInfty}
        \sup_n \|f_n\|_{B(\mathcal{X})} \leq C < \infty.
    \end{equation}
    Fix a sequence $\delta_n \rightarrow 0$. By Lemma \ref{lem:differenceOfKernelEstimate}, we know
    \begin{equation}
            \mathbb{P}
            \left(
                \bigcap_{i}
                \left\{
                    \limsup_{n\rightarrow \infty}
                    \sup_{x \in \mathcal{X}}
                    \left|
                    \int_\mathcal{X}
                    k_{\delta_i}(x;y)
                    \left(
                        d\mu_n(y)
                    -
                        d\mu(y)
                    \right)
                    \right|
                \leq
                    10\omega(2\delta_i)
                \right\}
            \right)
            =
            1.
        \end{equation}
        Similarly by Proposition \ref{prop:LawOfLargeNumber} we have
        \begin{equation}
            \mathbb{P}\left(
                \left\{
                    \lim_n
                    \|
                        d_{n,\mu} - d_{\mu}
                    \|_{B(\mathcal{X})}
                    =
                    0
                \right\}
            \right) = 1
        \end{equation}
        Thus for the rest of the proof we fix an event with probability $1$ in
        \begin{equation}
        \label{eq:prob1SetForEquicontinuity}
            \bigcap_{i}
                \left\{
                    \limsup_{n\rightarrow \infty}
                    \sup_{x \in \mathcal{X}}
                    \left|
                    \int_\mathcal{X}
                    k_{\delta_i}(x;y)
                    \left(
                        d\mu_n(y)
                    -
                        d\mu(y)
                    \right)
                    \right|
                \leq
                    10\omega(2\delta_i)
                \right\}
            \cap
            \left\{
                    \lim_n
                    \|
                        d_{n,\mu} - d_{\mu}
                    \|_{B(\mathcal{X})}
                    =
                    0
                \right\}.
        \end{equation}

        Under this assumptions we show that given any $\epsilon >0$ we can find $\delta > 0$ and $N \in \mathbb{N}$ such that for $n \geq N$
        \begin{equation}
            \left|
                    \hat{T}_{n,\mu}f_n(x)
                -
                    \hat{T}_{n,\mu}f_n(x')
                \right|
                <
                \epsilon.
        \end{equation}
        Let $\epsilon > 0$, and fix the constant $\alpha = \min\{\frac{a}{66C}, \frac{1}{9MC}\}> 0$.
        By equation \eqref{eq:UpAndLowBound} there exists $a>0$ such that $0 < a < d_{\mu}(x) = \int_\mathcal{X} k(x,y)d\mu(y)$. This fact and Proposition \ref{prop:imageInContinuousFunction} imply that $\frac{1}{d_{\mu}(x)} = \frac{1}{[P_\mu1](x)}$ is continuous, and since $\mathcal{X}$ is compact it is uniformly continuous. Thus there exists $\tilde{\delta} > 0$ such that if $d(x,x') < \tilde{\delta}$ we have
        \begin{equation}
        \label{eq:differenceOfAverages}
            \left|
                \frac{1}{d_{\mu}(x)}
            -
                \frac{1}{d_{\mu}(x')}
            \right|
        <
            \alpha
            \epsilon.
        \end{equation}
        Consider $l \in \mathbb{N}$ such that by equation \eqref{eq:maxDifferenceInBall} we can assume $\omega(\delta_l) \leq \omega(2\delta_l) < \alpha \epsilon$ and $\delta_l < \tilde{\delta}$. We take $\delta := \delta_l$. Furthermore since conditions \eqref{eq:prob1SetForEquicontinuity} are satisfied, there exists $N \in \mathbb{N}$ such that for all $n \geq N$ we have
        \begin{align}
        \label{eq:approximateDiffferenceKernelForEmpirical}
        &\left|
            \int_\mathcal{X}
            k_{\delta_l}(x;y)
            \left(
                d\mu_n(y)
            -
                d\mu(y)
            \right)
        \right|
        \leq
        10\omega(2\delta_l)
        <
            10\alpha\epsilon,\\
        &0<a<\inf_{x \in \mathcal{X}} d_{n,\mu}(x),
        \quad \quad
        \omega(\delta_l) < \alpha\epsilon,\\
        &\sup_{x \in \mathcal{X}}
        \left|
            \frac{1}{d_{\mu}(x)}
        -
            \frac{1}{d_{n,\mu}(x)}
        \right|
        <
            \min\{\alpha\epsilon, \frac{1}{2a}\}.
        \label{eq:averageConvergence}
        \end{align}
        Thus given $x, x' \in \mathcal{X}$ such that $d(x,x') < \delta = \delta_l < \tilde{\delta}$, by equation \eqref{eq:boundOnSequenceLInfty}
        \begin{align}
            &\left|
                \hat{T}_{n, \mu} f_n(x)
            -
                \hat{T}_{n, \mu} f_n(x')
            \right|\\
        \leq&
            \frac{1}{2}\int_{\mathcal{X}}
            \left|
            k(x,y)
            \left(
                \frac{1}
                {d_{n,\mu}(x)}
            +
                \frac{1}{d_{n,\mu}(y)}
            \right)f_n(y)
            -
                k(x',y)
                \left(
                    \frac{1}
                    {d_{n,\mu}(x')}
                +
                    \frac{1}
                    {d_{n,\mu}(y)}
                \right)f_n(y)
                \right|
                d\mu_n(y)\\
        \leq&\frac{C}{2}
            \int_\mathcal{X}
                \left|
                k(x,y)
                \left(
                    \frac{1}{d_{n,\mu}(x)}
                -
                    \frac{1}{d_{n,\mu}(x')}
                \right)
                \right|
                d\mu_n(y)
        +\frac{C}{2}
            \int_\mathcal{X}
                \left|
                    \left(
                        k(x,y)
                    -
                        k(x',y)
                    \right)
                    \frac{1}{d_{n,\mu}(x')}
                \right|
                d\mu_n(y)\\
        &\quad\quad+
            \frac{C}{2}\int_\mathcal{X}
                \left|
                    \left(
                        k(x,y)
                    -
                        k(x',y)
                    \right)
                    \frac{1}{d_{n,\mu}(y)}
                \right|
                d\mu_n(y).
            \label{eq:3PartTriangleIneq}
        \end{align}
        We bound each of these terms separately. Using the fact that $x' \in B_{\delta_l}(x)$, $\delta_l < \tilde{\delta}$ for $n \geq N$ and equations \eqref{eq:UpBoundKernel}, \eqref{eq:differenceOfAverages}, \eqref{eq:averageConvergence} we have
        \begin{align}
            &C
            \int_\mathcal{X}
                \left|
                k(x,y)
                \left(
                    \frac{1}{d_{n,\mu}(x)}
                -
                    \frac{1}{d_{n,\mu}(x')}
                \right)
                \right|
                d\mu_n(y)\\
            \leq&
                CM
                \left|
                    \frac{1}{d_{n,\mu}(x)}
                -
                    \frac{1}{d_{\mu}(x)}
                \right|+
                \left|
                    \frac{1}{d_{\mu}(x)}
                -
                    \frac{1}{d_{\mu}(x')}
                \right|+
                \left|
                    \frac{1}{d_{\mu}(x')}
                -
                    \frac{1}{d_{n,\mu}(x')}
                \right|
            \leq
                MC
                \left(
                    3\alpha\epsilon
                \right).
            \label{eq:Triang1}
        \end{align}
        On the other hand using equations \eqref{eq:UpAndLowBound}, \eqref{eq:maxDifferenceInBall}, \eqref{eq:approximateDiffferenceKernelForEmpirical} and \eqref{eq:averageConvergence} since $x' \in B_{\delta_l}(x)$ we have for $n \geq N$
        \begin{align}
            C
            \int_\mathcal{X}
                \left|
                    \left(
                        k(x,y)
                    -
                        k(x',y)
                    \right)
                    \frac{1}{d_{n,\mu}(y)}
                \right|
                d\mu_n(y)
        &\leq
            \frac{2C}{a}
            \int_\mathcal{X}
                \left|
                    k(x,y)
                -
                    k(x',y)
                \right|
                d\mu_n(y)\\
        &\leq
            \frac{2C}{a}
            \int_\mathcal{X}
                k_{\delta_l}(x;y)
                d\mu_n(y)\\
        &\leq
            \frac{2C}{a}
            \int_\mathcal{X}
            k_{\delta_l}(x;y)
            d\mu(y)
        +
            \frac{2C}{a}
            \int_\mathcal{X}
            k_{\delta_l}(x;y)
            \left(
                d\mu_n(y)
            -
                d\mu(y)
            \right)\\
        &\leq
            \frac{2C}{a}
            \left(
                10\omega(2\delta_l)
            +
                \alpha\epsilon
            \right)
        \leq
            \frac{22C}{a}\alpha\epsilon.
            \label{eq:triang2}
        \end{align}
        Similarly we have
        \begin{align}
            C
            \int_{\mathcal{X}}
                \left|
                    \left(
                        k(x,y)
                    -
                        k(x',y)
                    \right)
                    \frac{1}{d_{n,\mu}(x')}
                \right|
                d\mu_n(y)
        &\leq
            \frac{22C}{a}\alpha\epsilon.
            \label{eq:triang3}
        \end{align}
        Using the definition of $\alpha$ and the previous inequalities \eqref{eq:3PartTriangleIneq}, \eqref{eq:Triang1}, \eqref{eq:triang2}, \eqref{eq:triang3}, we conclude
        \begin{equation}
            \left|
                \hat{T}_{n, \mu} f_n(x)
            -
                \hat{T}_{n, \mu} f_n(x')
            \right|
        \leq
            \epsilon.
        \end{equation}
\end{proof}

\begin{lemma}
\label{lem:finiteOperatorIsCompact}
    With probability $1$ there exists $N \in \mathbb{N}$ such that for all $n \geq N$ the operator $\hat{T}_{n, \mu} : B(\mathcal{X}) \rightarrow B(\mathcal{X})$ is well defined, compact and satisfies
    \begin{equation}
        \|
            \hat{T}_{n, \mu}
        \|_{B(\mathcal{X})}
    \leq
    \frac{M}{a}.
    \end{equation}
\end{lemma}
\begin{proof}
    Start by using Proposition \ref{prop:LawOfLargeNumber} applied with $g = 1$, and equation \eqref{eq:UpAndLowBound} to conclude with probability $1$ there exists $N \in \mathbb{N}$ such that for all $n \geq N$ we have
    \begin{equation}
        \inf_{x \in X}
            d_{n,\mu}(x)
        \geq
            \frac{1}{2a}.
    \end{equation}
    This makes the operator defined by \eqref{eq:THatDef} well defined.

    Consider a sequence $f_j \in B(\mathcal{X})$ such that $\|f_j\|_{B(\mathcal{X})} \leq 1$ for all $j \geq N$. Given the random variables $X_1, ..., X_n$, for each sequence $\{f_j(X_i)\}_{j \geq N} \subset \mathbb{R}$, since it is bounded, we can find a subsequence $j_m$, such that for all $i \in \{1,...,n\}$ there exist $\alpha_i \in [-1,1]$ such that
    \begin{equation}
        \lim_m f_{j_m}(X_{i}) = \alpha_i.
    \end{equation}
    Consider the function given by
    \begin{equation}
        g(x)
    :=
        \frac{1}{n}
        \sum_{i=1}^n
        k(x,X_i)
        \left(
            \frac{1}{d_{n,\mu}(x)}
        +
            \frac{1}{d_{n,\mu}(X_i)}
        \right)
        \alpha_i
    \end{equation}
    Thus using the definition \eqref{eq:THatDef} we have
    \begin{align}
        \left|
            \hat{T}_{n, \mu} f_{j_m}(x)
        -
            g(x)
        \right|
    &\leq
        \left|
            \frac{1}{n}
            \sum_{i=1}^n
            \frac{1}{2}
            k(x,X_i)
        \left(
            \frac{1}{d_{n,\mu}(x)}
        +
            \frac{1}{d_{n,\mu}(X_i)}
        \right)
        \left(
            f_{j_m}(X_i)
        -
            \alpha_i
        \right)
        \right|\\ 
    &\leq
        \frac{1}{n}
        \sum_{i=1}^n
        \frac{2M}{a}
        \left|
            f_{j_m}(X_i)
        -
            \alpha_i
        \right|
        \rightarrow 0,
    \end{align}  
    where the right hand side converges uniformly to zero as $m \rightarrow \infty$, showing that $\|\hat{T}_{n, \mu}f_{j_m}-g\|_{B(\mathcal{X})}$ converges to zero, and so the operator $\hat{T}_{n, \mu}$ is compact. Finally we have for $n \geq N$ that given $f \in B(\mathcal{X})$ we have
    \begin{align}
        \|
            \hat{T}_{n, \mu} f
        \|_{B(\mathcal{X})}
    &\leq
        \sup_{x \in X}
        \left|
            \frac{1}{n}
            \sum_{i=1}^n
            \frac{1}{2}
            k(x,X_i)
        \left(
            \frac{1}{d_{n,\mu}(x)}
        +
            \frac{1}{d_{n,\mu}(X_i)}
        \right)
            f(X_i)
        \right|\\
    &\leq
        \sup_{x \in X}
        \frac{2M}
        {a}
        \|f\|_{B(\mathcal{X})}.
    \end{align}     
\end{proof}

\begin{prop}
\label{prop:collectiveCompactConvergence}
    Almost surely we have that $\hat{T}_{n,\mu}$ compactly converges to $T_\mu$, that is
    \begin{equation}
        \mathbb{P}\left(
            \left\{
                \lim_n
                \hat{T}_{n, \mu}
                \stackunder{$\longrightarrow$}{\scriptsize{$cc$}}
                T_{\mu}
            \right\}
        \right)
        =
        1.
    \end{equation}
\end{prop}
\begin{proof}
    By Lemma \ref{lem:finiteOperatorIsCompact} with probability $1$, there exists $N$ such that for $n \geq N$ the operator $\hat{T}_{n, \mu}$ is compact and well defined.
    Since by Proposition \ref{prop:pointWiseConvergence} we have that $\hat{T}_{n, \mu}\stackunder{$\longrightarrow$}{$\scriptsize{\text{p}}$} T_{\mu}$, we only need to show that the set
    \begin{equation}
        \bigcup_{n \geq N}\hat{T}_{n, \mu}\left(\left\{f\in B(\mathcal{X}): \|f\|_{B(\mathcal{X})}\leq 1\right\}\right)
    \end{equation}
    is compact. For that consider $f_l \in B(\mathcal{X})$ a sequence such that $\|f_l\|_{B(\mathcal{X})} \leq 1$. Let $\hat{T}_{n_l, \mu}$ a sequence of operators with $n_l \geq N$. If $n_l$ is bounded, then there is a subsequence $l_j \subset \mathbb{N}$ such that $n_{l_j} = m \in \mathbb{N}$ is constant. By Lemma \ref{lem:finiteOperatorIsCompact}, $\hat{T}_m$ is compact, thus $\hat{T}_{n_{l_j}, \mu} f_{l_j}$ has a convergent subsequence. Thus we can assume without loss of generality that our sequence is of the form $\hat{T}_{l,\mu} f_l$.
    
    Consider a sequence $\epsilon_n \rightarrow 0$. Using Lemma \ref{lem:equicontinuity}, we know that with probability $1$, we can find $N_n \geq N$ and $\delta_n > 0 $ such that for $l \geq N_n$
    \begin{equation}
    \label{eq:diffInBallEquicont}
        \sup_{x\in \mathcal{X}}
        \sup_{x' \in B_{\delta_n}(x)}
        \left|
            \hat{T}_{l,\mu}f_l(x)
        -
            \hat{T}_{l,\mu}f_l(x')
        \right|
        <
        \epsilon_n.
    \end{equation}
    Since $\mathcal{X}$ is compact, for each $n \in \mathbb{N}$, we can take a finite cover of $\mathcal{X}$ by balls $\{B_{\delta_n}(x_i^n)\}_{i \in \{1,...,L_n\}}$. Let
    \begin{equation}
        Y = \bigcup_{n}\{x_i^n\}_{i\in\{1,...,L_n\}} \subset \mathcal{X}.
    \end{equation}
    Since $Y$ is countable and $\hat{T}_{l,\mu}f_l$ is uniformly bounded (using Lemma \ref{lem:finiteOperatorIsCompact} and $\|f\|_{B(\mathcal{X})} \leq 1$), we can take a subsequence labeled in $j$, $\hat{T}_{l_j,\mu}f_{l_j}$ using a diagonal argument, such that
    \begin{equation}
    \label{eq:convergenceOnDenseSet}
        \forall {y \in Y}
        \quad
        \hat{T}_{l_j,\mu}f_{l_j}(y)
        \text{ converges}.
    \end{equation} 
    Now we show that $\hat{T}_{l_j,\mu}f_{l_j}$ is a Cauchy sequence in $B(\mathcal{X})$, which is Banach space showing convergence, which concludes the proof. For simplicity in notation we will consider the subsequence $\hat{T}_{l_j,\mu}f_{l_j}$ simply as $\hat{T}_{l,\mu} f_l$.

    Given the previous sequence $\epsilon_n > 0$, use equation \eqref{eq:convergenceOnDenseSet} to find $\tilde{N}_n \geq N_n$ such that for $\sigma,j \geq \tilde{N}_n$ we have
    \begin{equation}
    \label{eq:convergenceInCountablePoints}
        \left|\hat{T}_{\sigma,\mu}f_\sigma(x_i^n) - \hat{T}_{j,\mu} f_j(x_i^n)\right|
    <
        \epsilon_n
    \end{equation}
    for all $i \in \{1,...,L_n\}$ where $\{x_i^n\}_{i=1,...,L_n}$ are the centers associated with the cover of balls with radius $\delta_n$. Thus given $z \in X$, we have that $z \in B_{\delta_n}(x_i^n)$ for some $i \in \{1,...,L_n\}$. Thus using equations \eqref{eq:diffInBallEquicont} and \eqref{eq:convergenceInCountablePoints}, for $\sigma,j \geq \tilde{N}_n$
    \begin{align}
        \left|
            \hat{T}_{j,\mu}f_j(z) - \hat{T}_{\sigma,\mu}f_\sigma(z)
        \right|
    &\leq
        \left|
            \hat{T}_{j,\mu}f_j(z) - \hat{T}_{j,\mu}f_j(x_i^n)
        \right|
    +
        \left|
            \hat{T}_{j,\mu}f_j(x_i^n) - \hat{T}_{\sigma,\mu}f_\sigma(x_i^n)
        \right|
    +
        \left|
            \hat{T}_{\sigma,\mu}f_\sigma(x_i^n) - \hat{T}_{\sigma,\mu}f_\sigma(z)
        \right|\\
    &\leq
        3\epsilon_n.
    \end{align}
    Since this is true for all $\epsilon_n$ and $\epsilon_n \rightarrow 0$, this proves that $\hat{T}_{j,\mu}f_j$ is a Cauchy sequence finishing the proof.
\end{proof}

\begin{proof}[Proof of Theorem \ref{thm:llneigenfunctions1}]
    By Proposition \eqref{prop:collectiveCompactConvergence}, we know that $\hat{T}_{n, \mu}$ converges collectively compactly almost surely. This implies that $U_{n, \mu}' = I - \hat{T}_{n, \mu}$ converges compactly to $U_{\mu}'$. Since $\lambda \neq 1$ we conclude that $\lambda$ is an isolated eigenvalue of $U_{\mu}'$, since $T_{\mu}$ is compact, and $\sigma_{\textup{ess}}(U_{\mu}') = \sigma_{\textup{ess}}(I - T_{\mu}) = \sigma_{\textup{ess}}(I) = \{1\}$. Thus we can use Theorem \ref{thm:funcAnalAbstractSpectralConvergence} to conclude the result.
\end{proof}

\begin{proof}[Proof of Theorem \ref{thm:llneigenfunctions2}]
    By Proposition \eqref{prop:collectiveCompactConvergence}, we know that $\hat{T}_{n, \mu}$ converges collectively compactly almost surely. Also by Lemma \ref{prop:pointWiseConvergence}, we can conclude that $m_{n,\mu} = T_{n, \mu}1$ converges uniformly to $m_\mu = T_{\mu}1$ almost surely. This implies that the multiplication operators $M_{m_n}$ converge in the operator topology almost surely to $M_{m_\mu}$. Both of these convergences imply compact convergence that is closed under addition, thus $U_{n, \mu} = M_{m_{n,\mu}} - \hat{T}_{n, \mu}$ converges compactly to $U$. Since $\lambda \notin rg(m_\mu)$ and $T_{\mu}$ is compact, we have that
    $$
        \sigma_{\textup{ess}}(U_{\mu}) = \sigma_{\textup{ess}}(M_{m_\mu} - T_{\mu}) = \sigma_{\textup{ess}}(M_{m_\mu}) \subset rg(m_\mu),
    $$ 
    and so we have that $\lambda$ is an isolated eigenvalue of $U$. Thus using Theorem \ref{thm:funcAnalAbstractSpectralConvergence} we can conclude the rest of the results.
\end{proof}

\section{Rates of convergence}

In this section we obtain rates of convergence for the eigenvectors. To do this the main functional analysis tool we will use is Theorem \ref{thm:quantifiedSpectralConvergence}. We notice that to apply this theorem we need compact collective convergence, and so this cannot be applied directly to $U_{n, \mu}$ or even $U_{n, \mu}'$. However, since $U'_{n,\mu} = I - \hat{T}_{n,\mu}$ and $U'_\mu = I - T_{\mu}$, we have that the eigenvectors of $U'_{n,\mu}$ and $U_\mu'$ are the same as those of $\hat{T}_{n,\mu}$ and $\hat{T}_\mu$ respectively. Thus if we show rates for the convergence of the eigenvectors of $\hat{T}_{n,\mu}$ to those of $T_{\mu}$, we will be showing rates for the convergence of the eigenvectors of $U_{n,\mu}'$ to those of $U_{\mu}'$. From Proposition \ref{prop:collectiveCompactConvergence} we know that $\hat{T}_{n,\mu}$ collectively compactly converges to $T_\mu$, thus we can apply Theorem \ref{thm:quantifiedSpectralConvergence} to obtain a bound on the convergence between the eigenvectors of $\hat{T}_{n,\mu}$ and $T_\mu$. This is illustrated in Theorem \ref{thm:rateOfConvergenceTheorem} which we prove at the end of this section.

We start by making an estimation on some quantities that will help us bound the terms found in inequality \eqref{eq:estimationOfEigenfunctions}. This first estimation follows closely the work of \cite{LuxburgSpectralConsistency}.

\begin{lemma}
\label{lem:boundsOnDifferencesOfFunctionals}
    Given $g \in B(\mathcal{X})$ let $\tilde{g}(y) = \frac{g(y)}{d_{\mu}(y)}$. If $\inf_{x \in \mathcal{X}} d_{n,\mu}(x) \geq \frac{a}{2}$ then
    \begin{align}
        \|
            (T_{n, \mu}
        -
            T_{\mu})g
        \|_{B(\mathcal{X})}
        &\leq
        \frac{1}{a}
        \sup_{x \in \mathcal{X}}
        \left|
            P_{n, \mu}g(x)
        -
            P_{\mu}g(x)
        \right|
        +
        \frac{1}{2}
        \sup_{x \in \mathcal{X}}
        \left|
            P_{n, \mu}\tilde{g}(x)
        -
            P_{\mu}\tilde{g}(x)
        \right|,
        \\
        \|
            (T_{\mu}-T_{n, \mu})T_{n, \mu}
        \|_{B(\mathcal{X})}
        &\leq
        \sup_{x,y \in \mathcal{X}}
        \left|
            [T_{n,\mu}h_\mu(y,\cdot)](x)
        -
            [T_\mu h_\mu(y,\cdot)](x)
        \right|
        ,\\
        \|\hat{T}_{n, \mu} - T_{n, \mu}\|_{B(\mathcal{X})}
    &\leq
        \frac{4M}{a^2}
        \sup_{x \in \mathcal{X}}
        \left|
            [P_{n, \mu} 1](x)
        -
            [P_{\mu} 1](x)
        \right|.
    \end{align}
\end{lemma}
\begin{proof}
    The first inequality is a consequence of the definitions of $T_{n, \mu}$, $P_{n,\mu}$, $P_\mu$ and the bound \eqref{eq:UpAndLowBound}. That is given $g \in B(\mathcal{X})$ we have
    \begin{align}
        \|(T_{n,\mu}-T_\mu)g\|_{B(\mathcal{X})}
    &=
        \sup_{x \in \mathcal{X}}
        \left|
            \int_{\mathcal{X}}
                h_\mu(x,y)
                g(y)
                \left(
                    d\mu_n(y)
                -
                    d\mu(y)
                \right)
        \right|\\
    &=
        \frac{1}{2}
        \sup_{x \in \mathcal{X}}
        \left|
            \int_{\mathcal{X}}
                \left(
                    \frac{k(x,y)}{d_\mu(x)}
                +
                    \frac{k(x,y)}{d_\mu(y)}
                \right)
                g(y)
                \left(
                    d\mu_n(y)
                -
                    d\mu(y)
                \right)
        \right|\\
    &\leq
        \frac{1}{2}
        \sup_{x \in \mathcal{X}}
        \left|
            \frac{1}{d_\mu(x)}
            \int_{\mathcal{X}}
                    k(x,y)
                g(y)
                \left(
                    d\mu_n(y)
                -
                    d\mu(y)
                \right)
        \right|
        +
        \frac{1}{2}
        \sup_{x \in \mathcal{X}}
        \left|
            \int_{\mathcal{X}}
                    \frac{k(x,y)}{d_\mu(y)}
                g(y)
                \left(
                    d\mu_n(y)
                -
                    d\mu(y)
                \right)
        \right|\\
    &\leq
        \sup_{x \in \mathcal{X}}
        \left|
            \int_{\mathcal{X}}
                    k(x,y)
                \frac{g(y)}{a}
                \left(
                    d\mu_n(y)
                -
                    d\mu(y)
                \right)
        \right|
    +
        \frac{1}{2}
        \sup_{x \in \mathcal{X}}
        \left|
            \int_{\mathcal{X}}
                k(x,y)
                \frac{g(y)}{d_\mu(y)}
                \left(
                    d\mu_n(y)
                -
                    d\mu(y)
                \right)
        \right|
    \\
    &=
        \frac{1}{a}
        \sup_{x \in \mathcal{X}}
        \left|
            P_{n,\mu}g(x)
        -
            P_{\mu}g(x)
        \right|
        +
        \frac{1}{2}
        \sup_{x \in \mathcal{X}}
        \left|
            P_{n,\mu}\tilde{g}(x)
        -
            P_{\mu}\tilde{g}(x)
        \right|
        .
    \end{align}

    For the second inequality consider $g \in B(\mathcal{X})$ such that $\|g\|_{B(\mathcal{X})} \leq 1$, then
    \begin{align}
        \|
            (T_{\mu}-T_{n, \mu})T_{n, \mu}
            g
        \|_{B(\mathcal{X})}
    &\leq
        \sup_{x \in \mathcal{X}}
        \left|
            (T_{\mu}-T_{n, \mu})\left( \int_\mathcal{X} h_{\mu}(\cdot,y)g(y)d\mu_n(y) \right)(x)
        \right|\\
    &=
        \sup_{x \in \mathcal{X}}
        \left|
            \int_\mathcal{X}
            h_{\mu}(x,z)
            \left( \int_\mathcal{X} h_{\mu}(z,y)g(y)d\mu_n(y) \right)
            \left(
                d\mu(z)
            -
                d\mu_n(z)
            \right)
        \right|\\
    &=
        \sup_{x \in \mathcal{X}}
        \left|
            \int_\mathcal{X}
            g(y)
            \left( \int_\mathcal{X} h_{\mu}(z,y)
                h_{\mu}(x,z)
                \left(
                d\mu(z)
            -
                d\mu_n(z)
            \right)
            \right)
            d\mu_n(y)
        \right|\\
    &\leq
        \sup_{x \in \mathcal{X}}
        \|g\|_{B(\mathcal{X})}
        \sup_{y \in \mathcal{X}}
        \left|
            \int_\mathcal{X} 
                h_{\mu}(z,y)
                h_{\mu}(x,z)
                \left(
                d\mu(z)
            -
                d\mu_n(z)
            \right)
        \right|\\
    &
        \leq
        \sup_{x,y \in \mathcal{X}}
        \left|
            [T_{n,\mu}h_\mu(y,\cdot)](x)
        -
            [T_\mu h_\mu(y,\cdot)](x)
        \right|.
    \end{align}

    Finally given $g \in B(\mathcal{X})$ such that $\|g\|_{B(\mathcal{X})} \leq 1$, using equations \eqref{eq:UpBoundKernel}, \eqref{eq:UpAndLowBound}, we conclude
    \begin{align}
        \|
            (\hat{T}_{n, \mu} - T_{n, \mu})g
        \|_{B(\mathcal{X})}
    &=
        \frac{1}{2}
        \sup_{x \in \mathcal{X}}
        \left|
            \int_\mathcal{X}
                k(x,y)g(y)
                \left(
                    \frac{1}{d_{n,\mu}(x)}
                +
                    \frac{1}{d_{n,\mu}(y)}
                -
                    \frac{1}{d_{\mu}(x)}
                -
                    \frac{1}{d_{\mu}(y)}
                \right)
            d\mu_n(y)
        \right|\\
    &\leq
        \frac{1}{2}
        \sup_{x \in \mathcal{X}}
        \left|
            \int_\mathcal{X}
                k(x,y)g(y)
                \left(
                    \frac{d_{\mu}(y)-d_{n,\mu}(y)}{d_{n,\mu}(y)d_{\mu}(y)}
                +
                    \frac{d_{\mu}(x)-d_{n,\mu}(x)}{d_{n,\mu}(x)d_{\mu}(x)}
                \right)
            d\mu_n(y)
        \right|\\
    &\leq
        \frac{4M}{a^2}
        \sup_{x \in \mathcal{X}}
        \left|
            d_{\mu}(x)
            -
            d_{n,\mu}(x)
        \right|\\
    &=
        \frac{4M}{a^2}
        \sup_{x \in \mathcal{X}}
        \left|
            [P_{n, \mu} 1](x)
        -
            [P_{\mu} 1] (x)
        \right|.
    \end{align}
\end{proof}
\begin{remark}
The condition $\inf_{x \in \mathcal{X}} d_{n,\mu}(x) \geq \frac{a}{2}$ is necessary since we don't assume a bound from below on the kernel $k$, and only an integral bound. Thus to ensure that $d_{n,\mu}(x) \geq \frac{a}{2}$, we may need a large number of points if $k$ is zero on a large part of the set (for example $k(x,y) = \chi_{B_r(x)}(y)$ where $r$ is a small radius).
\end{remark}
\begin{prop}
\label{prop:estimateOnProjectionOfEigenvector}
    Suppose $\inf_{x \in \mathcal{X}} d_{n,\mu}(x) \geq \frac{a}{2}$. Let $u$ be an eigenvector of $T_\mu$ with eigenvalue $\lambda \neq 0$ of satisfying $\|u\|_{B(\mathcal{X})} = 1$ and $V \subset \mathbb{C}$ such that $V \cap \sigma(T_\mu) = \{\lambda\}$. Let $\text{Pr}_n$ be the spectral projection of $\sigma(\hat{T}_{n,\mu}) \cap V$. Then there exists a constant $\overline{C}_{\lambda, T_\mu}:=C_0\left(\frac{3}{2}+\frac{1}{a}+\frac{16M^2}{a^3}+\frac{4M}{a^2}\right)>0$ (depending on $\lambda$, $\sigma(T_\mu)$, $a$ and $M$) such that almost surely
    \begin{align}
    \label{eq:boundOnDifferenceToProjection}
        \|u - \text{Pr}_n u\|_{B(\mathcal{X})}
        &\leq
        \overline{C}_{\lambda, T_\mu}
        \Bigg(
        \sup_{x \in \mathcal{X}}
        \left|
            P_{n, \mu}\tilde{u}(x)
        -
            P_{\mu}\tilde{u}(x)
        \right|
        +
        \sup_{x,z \in \mathcal{X}}
        \left|
            [P_{n, \mu} h_{\mu}(z,\cdot)](x)
        -
            [P_{\mu} h_{\mu}(z, \cdot)](x)
        \right|\\
        &\quad\quad\quad
            +
        \sup_{x \in \mathcal{X}}
        \left|
            [P_{n, \mu} 1](x)
        -
            [P_{\mu} 1](x)
        \right|
        +
        \sup_{x \in \mathcal{X}}
        \left|
            P_{n, \mu}u(x)
        -
            P_{\mu}u(x)
        \right|
        \Bigg).
    \end{align}
\end{prop}
\begin{proof}
    Since $T_\mu$ is a compact operator and $\lambda \neq 0$ is an eigenvalue, it is of finite multiplicity. 
    Also by Proposition \ref{prop:collectiveCompactConvergence} we have compact collective convergence almost surely, thus we can apply Theorem \ref{thm:quantifiedSpectralConvergence} to conclude there exists $C_0>0$ such that
    \begin{equation}
    \label{eq:functionalAnalisysEq1}
        \|u - \text{Pr}_n u\|_{B(\mathcal{X})}
    \leq
        C_0
        \left(
            \|(\hat{T}_{n, \mu} - T_{\mu})u\|_{B(\mathcal{X})}
        +
            \|u\|_{B(\mathcal{X})}
            \|
                (T_{\mu} - \hat{T}_{n, \mu})\hat{T}_{n, \mu}
            \|_{B(\mathcal{X})}
        \right).
    \end{equation}
    We bound each of these terms individually. Notice that
    \begin{equation}
        \|T_{\mu}\|_{B(\mathcal{X})}\leq
        \frac{M}{a},
        \quad\quad
        \|T_{n,\mu}\|_{B(\mathcal{X})}\leq \frac{M}{a},
        \quad\quad
        \|\hat{T}_{n,\mu}\|_{B(\mathcal{X})} \leq \frac{2M}{a}.
    \end{equation}
    We can apply triangle inequality, the previous inequality and Lemmas \ref{lem:finiteOperatorIsCompact}, \ref{lem:boundsOnDifferencesOfFunctionals} to conclude
    \begin{align}
        \|
            (T_{\mu}
        -
            \hat{T}_{n, \mu}
            )
            \hat{T}_{n, \mu}
        \|_{B(\mathcal{X})}
    &\leq
           \|
            T_{\mu}
            T_{n, \mu}
        -
            T_{\mu}
            \hat{T}_{n, \mu}
        \|_{B(\mathcal{X})}
        +
        \|
            T_{\mu} T_{n, \mu}
        -
            T_{n, \mu} T_{n, \mu}
        \|_{B(\mathcal{X})}\\
    &\quad\quad
        +
        \|
            T_{n, \mu}T_{n, \mu}
        -
            T_{n, \mu}\hat{T}_{n, \mu}
        \|_{B(\mathcal{X})}
        +
        \|
            T_{n, \mu}\hat{T}_{n, \mu}
        -
            \hat{T}_{n, \mu}\hat{T}_{n, \mu}
        \|_{B(\mathcal{X})}\\
    &\leq
        \|T_{\mu}\|_{B(\mathcal{X})}
        \|T_{n, \mu} - \hat{T}_{n, \mu}\|_{B(\mathcal{X})}
    +
        \|(T_{\mu}-T_{n, \mu})T_{n, \mu}\|_{B(\mathcal{X})}
    \\
    &\quad\quad+
        \|T_{n, \mu}\|_{B(\mathcal{X})}
        \|T_{n, \mu} - \hat{T}_{n, \mu}\|_{B(\mathcal{X})}
    +
        \|\hat{T}_{n, \mu}\|_{B(\mathcal{X})}
        \|T_{n, \mu} - \hat{T}_{n, \mu}\|_{B(\mathcal{X})}\\
    &\leq
        \|(T_{\mu}-T_{n, \mu})T_{n, \mu}\|_{B(\mathcal{X})}
    +
        \frac{4M}{a}
        \|(T_{n,\mu}-\hat{T}_{n, \mu})\|_{B(\mathcal{X})}\\
    &\leq
        \sup_{x,y \in \mathcal{X}}
        \left|
            [T_{n,\mu}h_\mu(y,\cdot)](x)
        -
            [T_\mu h_\mu(y,\cdot)](x)
        \right|\\
    &\quad\quad
        +\frac{16M^2}{a^3}
        \sup_{x \in \mathcal{X}}
        \left|
            [P_{n, \mu} 1](x)
        -
            [P_{\mu} 1](x)
        \right|.
        \label{eq:bound1ForAF}
    \end{align}
    Also applying Lemma \ref{lem:boundsOnDifferencesOfFunctionals}
    
    \begin{align}
        \|(\hat{T}_{n, \mu} - T)u\|_{B(\mathcal{X})}
    &\leq
        \|u\|_{B(\mathcal{X})}
        \|
            \hat{T}_{n, \mu} - T_{n, \mu}
        \|_{B(\mathcal{X})}
    +
        \|
            (T_{n, \mu} - T_{\mu})u
        \|_{B(\mathcal{X})}\\
    &\leq
        \frac{4M}{a^2}
        \sup_{x \in \mathcal{X}}
        \left|
            [P_{n, \mu} 1](x)
        -
            [P_{\mu} 1](x)
        \right|
    +
        \frac{1}{2}
        \sup_{x \in \mathcal{X}}
        \left|
            P_{n, \mu}\tilde{u}(x)
        -
            P_{\mu}\tilde{u}(x)
        \right|
        \label{eq:bound2ForAF}\\
    &\quad\quad
        +
        \frac{1}{a}
        \sup_{x \in \mathcal{X}}
        \left|
            P_{n, \mu}u(x)
        -
            P_{\mu}u(x)
        \right|,
    \end{align}
    where $\tilde{u}(y) = \frac{1}{d_\mu(y)}u(y)$. Using equations \eqref{eq:functionalAnalisysEq1}, \eqref{eq:bound1ForAF}, \eqref{eq:bound2ForAF} and Lemma \eqref{lem:boundsOnDifferencesOfFunctionals} we conclude the proof.
\end{proof}

Now we are going to bound the terms of equation \eqref{eq:boundOnDifferenceToProjection}. In \cite{LuxburgSpectralConsistency} this was done using covering numbers (see definition \ref{def:coveringNumbers}). In particular there is \cite[Theorem 19]{LuxburgSpectralConsistency} (this Theorem can be obtained by using \cite[Theorem 2.34]{Mendelson} and \cite[Section 3.4]{Anthony}), which can bound the terms in equation \eqref{eq:boundOnDifferenceToProjection} using $N(\mathcal{F}, \epsilon, L^2(\mathcal{X},\mu_n))$ with high probability. However the authors in \cite{LuxburgSpectralConsistency} to bound the term $N(\mathcal{F}, \epsilon, L^2(\mathcal{X},\mu_n))$ use $N(\mathcal{F}, \epsilon, L^\infty(\mathcal{X}))$ for the family of functions $\mathcal{F} = \mathcal{K}, g\cdot \mathcal{H}, \mathcal{H} \cdot \mathcal{H}$. However in our situation, since we don't assume the kernel $k$ to be continuous, we might have $N(\mathcal{F}, \epsilon, L^\infty(\mathcal{X})) = \infty$ (consider for example $k(x,y) = \chi_{B_r(x)}(y)$ and $\epsilon < 1$). In fact condition \eqref{eq:maxDifferenceInBall} would give us finite $N(\mathcal{F}, \epsilon, L^1(\mathcal{X},\mu))$, which can not in general bound $N(\mathcal{F}, \epsilon, L^2(\mathcal{X},\mu_n))$ or even $N(\mathcal{F}, \epsilon, L^1(\mathcal{X},\mu_n))$. We decide to show a convergence rate directly using Lemmas \ref{lem:boundOnIndivProbTerm} and \ref{lem:probBoundOnKernelWithKernel}. This technique obtains rates of convergence of the order $O(\frac{\sqrt{\ln(n)}}{\sqrt{n}})$ which are slightly weaker than $O(\frac{1}{\sqrt{n}})$ which are the rates obtained in \cite{LuxburgSpectralConsistency}.

\begin{prop}
\label{prop:estimateOnProbabilityForRate}
Let $\gamma = 32M^2/a^2 + \frac{8}{3}M/a$. Consider $\alpha \geq 1$, a sequence $\delta_n > 0$ such that $\omega(\delta_n) \leq \frac{a}{8}\frac{\alpha}{\sqrt{n}}\sqrt{\ln( n)}, \tilde{\omega}(\delta_n) \leq \frac{a}{16M}\frac{\alpha}{\sqrt{n}}\sqrt{\ln( n)}$ and let $L_{\delta_n} = N(\mathcal{X},\delta_n,d)$ be the minimum number of balls of radius $\delta_n$ needed to cover $\mathcal{X}$. Given $u \in B(\mathcal{X})$ such that $\|u\|_{B(\mathcal{X})} = 1$, $N \in \mathbb{N}$, we have
    \begin{align}
    \label{eq:prob1}
        \mathbb{P}
        \left(
            \bigcap_{n \geq N}
            \left\{
                \sup_{x \in \mathcal{X}}
                \left|
                    P_{n, \mu} \tilde{u}(x)
                -
                    P_{\mu} \tilde{u}(x)
                \right|
        \leq
            \frac{\alpha}{\sqrt{n}}
            \sqrt{\ln(n)}
            \right\}
        \right)
    &\geq
        1
    -
         \sum_{n \geq N} 
            4L_{\delta_n}
            \frac{1}{n^{\frac{\alpha}{\gamma}}}\\
    \label{eq:probNew}
    \mathbb{P}
        \left(
            \bigcap_{n \geq N}
            \left\{
                \sup_{x \in \mathcal{X}}
                \left|
                    P_{n, \mu} u(x)
                -
                    P_{\mu} u(x)
                \right|
        \leq
            \frac{\alpha}{\sqrt{n}}
            \sqrt{\ln(n)}
            \right\}
        \right)
    &\geq
        1
    -
         \sum_{n \geq N} 
            4L_{\delta_n}
            \frac{1}{n^{\frac{\alpha}{\gamma}}}\\
    \label{eq:prob2}
    \mathbb{P}
        \left(
            \bigcap_{n \geq N}
            \left\{
                \sup_{x \in \mathcal{X}}
                \left|
                    [P_{n, \mu} 1](x)
                -
                    [P_{\mu} 1](x)
                \right|
        \leq
            \frac{\alpha}{\sqrt{n}}
            \sqrt{\ln(n)}
            \right\}
        \right)
    &\geq
            1
        -
            \sum_{n \geq N} 
            4L_{\delta_n}
            \frac{1}{n^{\frac{\alpha}{\gamma}}}\\
    \label{eq:prob3}
    \mathbb{P}
        \left(
            \bigcap_{n \geq N}
            \left\{
                \sup_{x, z \in \mathcal{X}}
                \left|
                    [T_{n, \mu} h_\mu(z,\cdot)](x)
                -
                    [T_{\mu} h_\mu(z,\cdot)](x)
                \right|
        \leq
            \frac{\alpha}{\sqrt{n}}
            \sqrt{\ln(n)}
            \right\}
        \right)
    &\geq
            1
        -
        \sum_{n \geq N} 
            4L_{\delta_n}^2\frac{1}{n^{\frac{\alpha a^2}{M^2\gamma}}}.
    \end{align}
    where $\tilde{u}(y) = \frac{1}{d_\mu(y)}u(y)$. If $\delta_a > 0$ is such that $w(\delta_a) \leq \frac{1}{16}a$, then
    \begin{equation}
    \label{eq:prob4}
        \mathbb{P}
        \left(
            \forall_{n \geq N}
            \left\{
                \sup_{x \in \mathcal{X}}
                \left|
                    P_{n, \mu} 1(x)
                -
                    P_{\mu} 1(x)
                \right|
        \leq
            \frac{a}{2}
            \right\}
        \right)
    \geq
        1
    -
        \sum_{n \geq N}
        4L_{\delta_a}
        \exp\left\{
            \frac{-n(a/2)^2}
            {32M^2 + \frac{8}{3}M(a/2)}
        \right\}
    \end{equation}
\end{prop}
\begin{proof}
    In this proof we make use of the fact that $0<a \leq 1 \leq M$. Notice that $\|\tilde{u}\|_{B(\mathcal{X})} \leq \frac{1}{a}$. Let $\delta_n$ be such that $\omega(\delta_n) \leq \frac{a}{8}\frac{\alpha}{\sqrt{n}}\ln( n) \leq \frac{1}{8\|\tilde{u}\|_{B(\mathcal{X})}} \frac{\alpha}{\sqrt{n}}\ln( n)$. Notice that $\max_{n\geq 1} \frac{\sqrt{\ln(n)}}{\sqrt{n}} \leq 1$. Then using Lemma \ref{lem:boundOnIndivProbTerm}, for $n \geq 1$ we have
    \begin{align}
        &\mathbb{P}
        \left(
            \forall_{n \geq N}
            \left\{
                \sup_{x \in \mathcal{X}}
                \left|
                    P_{n, \mu} \tilde{u}(x)
                -
                    P_{\mu} \tilde{u}(x)
                \right|
        \leq
            \frac{\alpha}{\sqrt{n}}
            \sqrt{\ln( n)}
            \right\}
        \right)\\
    =&
        1
    -
        \mathbb{P}
        \left(
            \bigcup_{n \geq N}
            \left\{
                \sup_{x \in \mathcal{X}}
                \left|
                    P_{n, \mu} \tilde{u}(x)
                -
                    P_{\mu} \tilde{u}(x)
                \right|
        >
            \frac{\alpha}{\sqrt{n}}
            \sqrt{\ln( n)}
            \right\}
        \right)\\
    \geq&
        1
    -
        \sum_{n \geq N}
            2L_{\delta_n}
            \left(
                \exp\left\{
                    \frac
                    {-n
                    \left(
                        \frac{\alpha}{\sqrt{n}}
                        \sqrt{\ln( n)}
                    \right)^2
                    }
                    {32M^2 + \frac{8}{3}M\left(
                        \frac{\alpha}{\sqrt{n}}
                        \sqrt{\ln( n)}
                    \right)}
                \right\}
            +
                \exp\left\{
                    \frac{-n\left(
                        \frac{\alpha}{\sqrt{n}}
                        \sqrt{\ln( n)}
                    \right)^2}
                {32\|\tilde{u}\|_{B(\mathcal{X})}^2M^2 + \frac{8}{3} \|\tilde{u}\|_{B(\mathcal{X})} M \left(
                        \frac{\alpha}{\sqrt{n}}
                        \sqrt{\ln( n)}
                    \right)}
                \right\}
            \right)\\
    \geq&
        1
    -
        \sum_{n \geq N}
            2L_{\delta_n}
            \left(
                \exp\left\{
                    \frac
                    {-n
                    \left(
                        \frac{\alpha}{\sqrt{n}}
                        \sqrt{\ln( n)}
                    \right)^2
                    }
                    {32M^2 + \frac{8}{3}M\left(
                        \frac{\alpha}{\sqrt{n}}
                        \sqrt{\ln( n)}
                    \right)}
                \right\}
            +
                \exp\left\{
                    \frac{-n\left(
                        \frac{\alpha}{\sqrt{n}}
                        \sqrt{\ln( n)}
                    \right)^2}
                {32M^2/a^2 + \frac{8}{3} M\left(
                        \frac{\alpha}{\sqrt{n}}
                        \sqrt{\ln( n)}
                    \right)/a}
                \right\}
            \right)\\
    \geq&
        1
    -
        \sum_{n \geq N}
        4L_{\delta_n}
                \exp\left\{
                    \frac
                    {-n
                    \left(
                        \frac{\alpha}{\sqrt{n}}
                        \sqrt{\ln( n)}
                    \right)^2
                    }
                    {32M^2/a^2 + \frac{8}{3}M\left(
                        \frac{\alpha}{\sqrt{n}}
                        \sqrt{\ln( n)}
                    \right)/a}
                \right\}
    \end{align}
    Take $\gamma = 32M^2/a^2 + \frac{8}{3}M/a$. Since $\alpha \geq 1$,
    \begin{align}
        &\mathbb{P}
        \left(
            \bigcap_{n \geq N}
            \left\{
                \sup_{x \in \mathcal{X}}
                \left|
                    P_{n, \mu} \tilde{u}(x)
                -
                    P_{\mu} \tilde{u}(x)
                \right|
        \leq
            \frac{\alpha}{\sqrt{n}}
            \sqrt{\ln( n)}
        \right\}
        \right)\\
    \geq&
        1
    -
        \sum_{n \geq N}
        4L_{\delta_n}
                \exp\left\{
                    \frac
                    {-n
                    \left(
                        \frac{\alpha}{\sqrt{n}}
                        \sqrt{\ln( n)}
                    \right)^2
                    }
                    {32M^2/a^2 + \frac{8}{3}M\alpha/a}
                \right\}
    \geq
        1
        -
        \sum_{n \geq N}
            4L_{\delta_n}
                \exp\left\{
                    \frac
                    {-
                        \alpha
                        \ln( n)
                    }
                    {32M^2/\alpha a^2 + \frac{8}{3}M/a}
                \right\}\\
    \geq&
        1 - \sum_{n \geq N} 
            4L_{\delta_n}
            \exp
            \left\{
                -\frac{\alpha}{\gamma}
                \ln(n)
            \right\}
    =
        1 - \sum_{n \geq N} 
            4L_{\delta_n}
            \exp
            \left\{
                -
                \ln(n^\frac{\alpha}{\gamma})
            \right\}\\
    =&
        1 - \sum_{n \geq N} 
            4L_{\delta_n}
            \exp
            \left\{
                -
                \ln(n^\frac{\alpha}{\gamma})
            \right\}.
    \end{align}
    Inequality \eqref{eq:probNew} and \eqref{eq:prob2} are proved similarly.
    Considering $\delta_n>0$ such that $\tilde{\omega}(\delta_n) \leq \frac{a}{16M}\frac{\alpha}{\sqrt{n}}\sqrt{\ln(n)}$, inequality \eqref{eq:prob3} can be proved similarly using Corollary \ref{cor:CrossedTermEstimate}
    \begin{align}
        &\mathbb{P}
        \left(
            \bigcap_{n \geq N}
            \left\{
                \sup_{x, z \in \mathcal{X}}
                \left|
                    [T_{n, \mu} h_\mu(z,\cdot)](x)
                -
                    [T_{\mu} h_\mu(z,\cdot)](x)
                \right|
        \leq
            \frac{\alpha}{\sqrt{n}}
            \sqrt{\ln(n)}
            \right\}
        \right)\\
    \geq&
        1
        -
        \sum_{n \geq N} 
            4L_{\delta_n}^2
        \exp
        \left(
            \frac{-na^2\left(\frac{\alpha}{\sqrt{n}}
            \sqrt{\ln(n)}\right)^2}
            {32M^4/a^2 + \frac{8}{3}M^2\left(\frac{\alpha}{\sqrt{n}}
            \sqrt{\ln(n)}\right)}
        \right)\\
    \geq&
        1
        -
        \sum_{n \geq N} 
            4L_{\delta_n}^2
        \exp
        \left(
            \frac{-na^2\left(\frac{\alpha}{\sqrt{n}}
            \sqrt{\ln(n)}\right)^2}
            {32M^4/a^2 + \frac{8}{3}M^2\left(\frac{\alpha}{\sqrt{n}}
            \sqrt{\ln(n)}\right)}
        \right)\\
    \geq&
        1
        -
        \sum_{n \geq N} 
            4L_{\delta_n}^2
        \exp
        \left(
            \frac{-na^2\left(\frac{\alpha}{\sqrt{n}}
            \sqrt{\ln(n)}\right)^2}
            {32M^4/a^2 + \frac{8}{3}M^3\alpha/a}
        \right)\\
    \geq&
        1
        -
        \sum_{n \geq N} 
            4L_{\delta_n}^2
        \exp
        \left(
            \frac{-\alpha a^2\ln(n)}
            {M^2\gamma}
        \right)
    =
        1
    -
         \sum_{n \geq N} 
            4L_{\delta_n}^2
        \exp
        \left(
            -\ln({n^{\frac{\alpha a}{M^2\gamma}}})
        \right)\\
    =&
        1
        -
        \sum_{n \geq N} 
            4L_{\delta_n}^2\frac{1}{n^{\frac{\alpha a^2}{M^2\gamma}}}.
    \end{align}
    For inequality \eqref{eq:prob4}, choosing $\delta_a>0$ such that $\omega(\delta_a)\leq \frac{a}{16}$ applying equation \eqref{eq:simplifiedBound} from Lemma \ref{lem:boundOnIndivProbTerm} we obtain
    \begin{align}
    \mathbb{P}
        \left(
            \forall_{n \geq N}
            \left\{
                \sup_{x \in \mathcal{X}}
                \left|
                    P_{n, \mu} 1(x)
                -
                    P_{\mu} 1(x)
                \right|
        \leq
            \frac{a}{2}
            \right\}
        \right)
    \geq&
        1
    -
        \sum_{n \geq N}
        \mathbb{P}
        \left(
            \left\{
                \sup_{x \in \mathcal{X}}
                \left|
                    P_{n, \mu} 1(x)
                -
                    P_{\mu} 1(x)
                \right|
        >
            \frac{a}{2}
            \right\}
        \right)\\
    \geq&
        1
    -
        \sum_{n \geq N}
        4L_{\delta_a}
        \exp\left\{
            \frac{-n(a/2)^2}
            {32M^2 + \frac{8}{3}M(a/2)}
        \right\}.
    \end{align}
\end{proof}
\begin{lemma}
    \label{lem:auxilliaryLemmaCalc1}
    Let $ 1 < \sigma$, and $B > 0$ and $N \in \mathbb{N}$ with $N \geq 2$. Then we have
    \begin{equation}
        \sum_{n \geq N}
        \frac{1}{n^\sigma}
    \leq
        \frac{1}{(\sigma - 1)(N-1)^{\sigma - 1}}
    \end{equation}
    and
    \begin{equation}
        \sum_{n \geq N}
        \exp
        \left\{
            \frac{-n(a/2)^2}{32M^2 + \frac{8}{3}M(a/2)}
        \right\}
    \leq
        \frac{1}{C_e}
        \exp
        \left\{
            -(N-1)C_e
        \right\}
    \end{equation}
    where $C_e = \frac
            {(a/2)^2}
            {32M^2 + \frac{8}{3}M(a/2)}$.
\end{lemma}
\begin{proof}

    We have
    \begin{align}
        \sum_{n\geq N}
        \frac{1}{n^\sigma}
    \leq&
        \int_{N-1}^\infty
            \frac{1}{x^\sigma}
            dx
    =
        \frac{1}{\sigma-1}
        \frac{1}{(N-1)^{\sigma-1}}.
    \end{align}
    Similarly
    \begin{align}
        \sum_{n \geq N}
        \exp\left\{
            \frac{-n^2(a/2)^2}{32M^2 + \frac{8}{3}M(a/2)}
        \right\}
    \leq&
        \int_{N-1}^\infty
        \exp
        \left\{
            \frac{-x(a/2)^2}{32M^2 + \frac{8}{3}M(a/2)}
        \right\}
        dx\\
    =&
        \frac{32M^2 + \frac{8}{3}M(a/2)}{(a/2)^2}
        \exp
        \left\{
            -\frac
            {-(N-1)(a/2)^2}
            {32M^2 + \frac{8}{3}M(a/2)}
        \right\}.
    \end{align} 
\end{proof}

\begin{proof}[Proof of Theorem \ref{thm:rateOfConvergenceTheorem}]
    We start by noticing that since $U_\mu' = I - T_\mu$, if $\lambda \neq 1$ is an eigenvalue of $U_\mu$  with eigenvector $u \in B(\mathcal{X})$ and $\sigma(U_\mu) \cap V = \{\lambda\}$, we have that the set $V_1 = \{1-x:x \in V\}$ satisfies $V_1 \cap \sigma(T_\mu) = \{1-\lambda\}$ and $u$ is an eigenvector of $T_\mu$ with eigenvalue $1-\lambda \neq 0$. If $\textup{Pr}_n$ is the spectral projection of $\sigma(U_{n,\mu}')\cap V$ and $\tilde{\text{Pr}}_n$ is the spectral projection of $\sigma(T_{n,\mu})\cap V_1$ then since $U_{n,\mu}' = I - \hat{T}_{n,\mu}$ we must have that $\text{Pr}_n = \tilde{\text{Pr}}_n$.

    We consider $\delta_a > 0$ such that $\omega(\delta_a) = \frac{a}{16}$. We can apply equation \eqref{eq:hypOnCoverAndContinuity} to conclude that
    \begin{equation}
        \delta_a
    \geq
    \left(
        \frac{\omega(\delta_a)}{C_\omega}
    \right)^\frac{1}{m'}
    =
    \left(
        \frac{a}{16C_\omega}
    \right)^\frac{1}{m'}.
    \end{equation}
    Thus we have
    \begin{equation}
        L_{\delta_a}
    \leq
        C_L
        \left(
            \left(
            \frac{a}{16C_\omega}
            \right)^\frac{1}{m'}
        \right)^{-m}
    =
        C_L
            \left(
            \frac{16C_\omega}{a}
            \right)^\frac{m}{m'}
    =
        :C_a.
    \end{equation}
    By Proposition \ref{prop:estimateOnProbabilityForRate} we obtain
    \begin{equation}
    \label{eq:probabilityOfPositiveKernel}
        \mathbb{P}
        \left(
            \bigcap_{n \geq N}
            \left\{
                \sup_{x \in \mathcal{X}}
                \left|
                    P_{n, \mu} 1(x)
                -
                    P_{\mu} 1(x)
                \right|
        \leq
            \frac{a}{2}
            \right\}
        \right)
    \geq
        1
    -
        \sum_{n \geq N}
        4C_a
        \exp\left\{
            \frac{-n(a/2)^2}
            {32M^2 + \frac{8}{3}M(a/2)}
        \right\}.
    \end{equation}

    Notice that by Lemma \ref{lem:estimatesForH} then $\tilde{\omega}(\delta) \leq \frac{a}{16M}\epsilon$ is equivalent to $\omega(\delta) \leq \frac{a}{8M}\left(\frac{1}{a}+\frac{2M}{a^2}\right)^{-1} \epsilon$. Let $C_\delta:= \min\{\frac{a}{8}, \frac{a}{8M}\left(\frac{1}{a}+\frac{2M}{a^2}\right)^{-1}\}=\frac{a}{8M}\left(\frac{1}{a}+\frac{2M}{a^2}\right)^{-1}$.
    Using continuity of $\omega(\cdot)$, choose $\delta_n$ such that
    \begin{equation}
        \omega(\delta_n) = C_\delta \frac{\alpha}{\sqrt{n}}\sqrt{\ln( n)} \leq 
        \min\left\{
            \frac{1}{8}\frac{\alpha}{\sqrt{n}}\ln( n), \frac{a}{8M}\left(\frac{1}{a}+\frac{2M}{a^2}\right)^{-1}
            \frac{\alpha}{\sqrt{n}}\sqrt{\ln( n)}
        \right\}. 
    \end{equation}
    We can use the hypothesis \eqref{eq:hypOnCoverAndContinuity} of the theorem to conclude
    \begin{align}
        \delta_n
    \geq
        \left(\frac{w(\delta_n)}{C_\omega}\right)^\frac{1}{m'}
    =
        \left(
            \frac{C_\delta}{C_\omega}
            \frac{\alpha}{\sqrt{n}}
            \ln(n)
        \right)^\frac{1}{m'}.
    \end{align}
    Again using equation \eqref{eq:hypOnCoverAndContinuity} we conclude
    \begin{equation}
        L_{\delta_n}
    \leq
        C_L\left(
            \frac{C_\delta}{C_\omega}
            \frac{\alpha}{\sqrt{n}}
            \ln(n)
        \right)^\frac{-m}{m'}
    =
        C_L
        \left(
            \frac{C_\omega}{C_\delta}
            \frac{\sqrt{n}}{\alpha\ln(n)}
        \right)^\frac{m}{m'}
    \leq
        \tilde{C}\frac{n^\frac{m}{2m'}}{\alpha^{\frac{m}{m'}}},
    \end{equation}
    where we take $\tilde{C}:=C_L\left(\frac{C_\omega}{C_\delta\ln(2)}\right)^\frac{m}{m'} = C_L\left(\frac{8MC_\omega}{a \ln(2) }\left(\frac{1}{a}+\frac{2M}{a} \right) \right)^\frac{m}{m'}$. Using $\frac{\alpha}{\tilde{\gamma}}
        -
        \frac{m}{2m'}
        >
        1$, we also have that $\alpha \geq 1$. Since $\omega(\delta_n) \leq \frac{a}{8}\frac{\alpha}{\sqrt{n}}\sqrt{\ln(n)}$, $\tilde{\omega}(\delta_n) \leq \frac{a}{16M}\frac{\alpha}{\sqrt{n}}\ln(n)$ we can now apply Proposition \ref{prop:estimateOnProbabilityForRate}, Lemma \ref{lem:auxilliaryLemmaCalc1} and equation \eqref{eq:probabilityOfPositiveKernel} to conclude that we have for all $n \geq N$
        \begin{align}
        \label{eq:aproximatingEq1}
        \sup_{x \in \mathcal{X}}
                \left|
                    P_{n, \mu} \tilde{u}(x)
                -
                    P_{\mu} \tilde{u}(x)
                \right|
        &\leq
            \frac{\alpha}{\sqrt{n}}
            \sqrt{\ln(n)},\\
        \label{eq:aproximatingEqNew}
                \sup_{x \in \mathcal{X}}
                \left|
                    P_{n, \mu} u(x)
                -
                    P_{\mu} u(x)
                \right|
        &\leq
            \frac{\alpha}{\sqrt{n}}
            \sqrt{\ln(n)}\\
        \label{eq:aproximatingEq2}
        \sup_{x \in \mathcal{X}}
                \left|
                    [P_{n, \mu} 1](x)
                -
                    [P_{\mu} 1](x)
                \right|
        &\leq
            \frac{\alpha}{\sqrt{n}}
            \sqrt{\ln(n)},\\
        \label{eq:aproximatingEq3}
        \sup_{x, z \in \mathcal{X}}
                \left|
                    [T_{n, \mu} h_{\mu}(z,\cdot)](x)
                -
                    [T_{\mu} h_{\mu}(z,\cdot)](x)
                \right|
        &\leq
            \frac{\alpha}{\sqrt{n}}
            \sqrt{\ln(n)},\\
        \label{eq:aproximatingEq4}
        \sup_{x \in \mathcal{X}}
                \left|
                    P_{n, \mu} 1(x)
                -
                    P_{\mu} 1(x)
                \right|
        &\leq
            \frac{a}{2},
    \end{align}
    is satisfied with probability larger than (using definition of $\tilde{\gamma}$ from Theorem \ref{thm:rateOfConvergenceTheorem} and $C_e$ from Lemma \ref{lem:auxilliaryLemmaCalc1}) 
    \begin{align}
        &
        1
    -
        \sum_{n \geq N}
        4L_{\delta_n} \frac{3}{n^{\frac{\alpha}{\gamma}}}
    -
            \sum_{n \geq N}
            4L_{\delta_n}^2
            \frac{1}{n^{\frac{\alpha a^2}{M^2\gamma}}}
    -
        \sum_{n \geq N}
        4C_a
        \exp\left\{
            \frac{-n^2(a/2)^2}
            {32M^2 + \frac{8}{3}Mn(a/2)}
            \right\}\\
    \geq&
        1
    -
        \sum_{n \geq N}
        4L_{\delta_n}^2 
        \frac{3}{n^{\frac{\alpha}{\gamma}}}
    -
            \sum_{n \geq N}
            2L_{\delta_n}^2
            \frac{1}{n^{\frac{\alpha a^2}{M^2\gamma}}}
    -
        \sum_{n \geq N}
        4C_a
        \exp\left\{
            \frac{-n(a/2)^2}
            {32M^2 + \frac{8}{3}M(a/2)}
            \right\}\\
    \geq&
        1
    -
        \sum_{n \geq N}
        12\left(\tilde{C}\frac{n^\frac{m}{2m'}}{\alpha^{\frac{m}{m'}}}\right)^2 \frac{1}{n^{\frac{\alpha}{\gamma}}}
    -
            \sum_{n \geq N} 
            4\left(\tilde{C}\frac{n^\frac{m}{2m'}}{\alpha^{\frac{m}{m'}}}\right)^2
            \frac{1}{n^{\frac{\alpha a^2}{M^2\gamma}}}
            -
        \sum_{n \geq N}
        4C_a
        \exp\left\{
            \frac{-n(a/2)^2}
            {32M^2 + \frac{8}{3}M(a/2)}
            \right\}\\
    \geq&
        1
        -
        \frac{\tilde{C}}
        {\alpha^{\frac{2m}{m'}}}
        \sum_{n\geq N}
        \left(  
            \frac
            {12}
            {n^{\frac{\alpha}{\gamma} - \frac{m}{m'}}}
        +
                \frac{4}
            {n^{\frac{\alpha}{M^2\gamma/a^2} - \frac{m}{m'}
            }
            }
        \right)
    -
        \sum_{n \geq N}
        4C_a
        \exp\left\{
            \frac{-n(a/2)^2}
            {32M^2 + \frac{8}{3}M(a/2)}
            \right\}\\
    \geq&
        1
        -
        \frac{\tilde{C}}
        {\alpha^{\frac{2m}{m'}}}
        \sum_{n\geq N}
            \frac{16}
            {n^{\frac{\alpha}{\tilde{\gamma}} - \frac{m}{m'}
            }
            }
    -
        \sum_{n \geq N}
        4C_a
        \exp\left\{
            \frac{-n(a/2)^2}
            {32M^2 + \frac{8}{3}M(a/2)}
            \right\}\\
    \geq&
        1
        -
        \frac{\tilde{C}}
        {\alpha^{\frac{2m}{m'}}}
        \sum_{n\geq N}
            \frac{16}
            {n^{\frac{\alpha}{\tilde{\gamma}} - \frac{m}{m'}
            }
            }
    -
        \sum_{n \geq N}
        4C_a
        \exp\left\{
            \frac{-n(a/2)^2}
            {32M^2 + \frac{8}{3}M(a/2)}
            \right\}\\
    \geq&
        1
        -\frac{16\tilde{C}}{\alpha^{\frac{2m}{m'}}}
        \frac{\tilde{\gamma}m'}
        {\alpha m' - \tilde{\gamma}m - \tilde{\gamma}m'}
        \frac{1}{(N-1)^{\frac{\alpha}{\tilde{\gamma}}-\frac{m}{2m'}}}
        +
        4\frac{C_a}{C_e}
        \exp
        \left\{
            -(N-1)C_e
        \right\}.
    \end{align}
    Finally we have that equation \eqref{eq:aproximatingEq4} implies that for all $n \geq N$ we have $\inf_{x \in \mathcal{X}} d_{n,\mu}(x) \geq \frac{a}{2}$, thus we can use equations \eqref{eq:aproximatingEq1}, \eqref{eq:aproximatingEqNew}, \eqref{eq:aproximatingEq2}, \eqref{eq:aproximatingEq3} and apply Proposition \ref{prop:estimateOnProjectionOfEigenvector} to conclude that 
    \begin{align}
        \|
            u - \text{Pr}_n u\|_{B(\mathcal{X})}
    &=
        \|u-\tilde{\text{Pr}}_nu\|_{B(\mathcal{X})}\\
        &\leq
        \overline{C}_{\lambda, T_\mu}
        \Bigg(
        \sup_{x \in \mathcal{X}}
        \left|
            P_{n, \mu}\tilde{u}(x)
        -
            P_{\mu}\tilde{u}(x)
        \right|
        +
        \sup_{x,z \in \mathcal{X}}
        \left|
            [
            T_{n, \mu} h_{\mu}(z,\cdot)](x)
        -
            [T_{\mu} h_{\mu}(z, \cdot)](x)
        \right|
        \\
        &\quad\quad\quad
            +
        \sup_{x \in \mathcal{X}}
        \left|
            [P_{n, \mu} 1](x)
        -
            [P_{\mu} 1](x)
        \right|
            +
            \sup_{x \in \mathcal{X}}
        \left|
            P_{n, \mu} u(x)
        -
            P_{\mu} u(x)
        \right|
        \Bigg)\\
    &\leq
        4\overline{C}_{\lambda, T_\mu}
        \frac{\alpha\sqrt{\ln(n)}}{\sqrt{n}}
    \end{align}
    finishing the proof.
\end{proof}

\begin{remark}
    Conditions \eqref{eq:hypOnCoverAndContinuity} are in general easily obtained and estimated. If we have a bounded set $\mathcal{X} \subset \mathbb{R}^N$ with the euclidean metric then in general we have a bound
    \begin{equation}
        L_\delta
    \leq
        C_L\frac{1}{\delta^N}.
    \end{equation}
    In fact if $\mathcal{X}$ is a compact submanifold of a given dimension $m$, then this bound can be improved to
    \begin{equation}
        L_\delta
    \leq
        C_L\frac{1}{\delta^m}.
    \end{equation}
    The bound on $w(\delta)$ however in general should not be better than linear on $\delta$. Indeed if our kernel is smooth, we are expecting a linear bound
    \begin{equation}
        \omega(\delta) \leq C_0\delta,
    \end{equation}
    where $C_0$ depends on the supremmum of the derivatives of $k$. Also if $k(x,y) = \chi_{B_r(x)}(y)$ then if $\mathcal{X}$ is a submanifold equipped with the Hausdorff measure, $\omega(\delta)$ as noted in remark \ref{rem:remarkAboutAnulli} this will correspond to the measure of an Annulli of width $\delta$ which is linear with $\delta$. In fact when $k$ is the characteristic function of the ball, the condition on $\omega$ is called the $m'$-annular decay property. In fact this property is studied in \cite{Buckley1999} and it is proved that length spaces that are doubling satisfy this property.
\end{remark}

\printbibliography

\end{document}